\documentclass[11pt, a4paper, UKenglish]{amsart}

\usepackage{amsmath,amssymb,amsfonts,amsthm}
\usepackage[utf8]{inputenc}
\usepackage{hyperref}
\usepackage{graphicx}
\usepackage{tikz-cd}
\usepackage[margin=1.2 in]{geometry}

\graphicspath{{figures/}}

\usepackage[style=alphabetic,sorting=nyt,maxbibnames=7,maxcitenames=5,block=space,backend=biber]{biblatex}
\usepackage{enumerate}

\newtheorem{thm}{Theorem}
\numberwithin{thm}{section}
\numberwithin{figure}{section}

\newtheorem{prop}[thm]{Proposition}
\newtheorem{lem}[thm]{Lemma}
\newtheorem{cor}[thm]{Corollary}

\theoremstyle{definition}
\newtheorem{dfn}[thm]{Definition}

\newtheorem{quest}[thm]{Question}
\newtheorem*{claim}{Claim}

\DeclareMathOperator{\h}{H}

\DeclareMathOperator{\TT}{T}

\newcommand{\bd}{\partial}
\newcommand{\cutalong}{\backslash\!\!\backslash}
\newcommand{\into}{\hookrightarrow}

\newcommand{\st}{\,|\,}
\newcommand{\iso}{\cong}
\renewcommand{\SS}{\mathbb{S}}
\newcommand{\DD}{\mathbb{D}}
\newcommand{\C}{\mathrm{C}}
\newcommand{\RR}{\mathbb{R}}
\newcommand{\QQ}{\mathbb{Q}}
\newcommand{\CC}{\mathbb{C}}
\newcommand{\NN}{\mathbb{N}}
\newcommand{\ZZ}{\mathbb{Z}}
\newcommand{\sub}{\subseteq}

\renewcommand{\S}{\mathcal{S}}
\newcommand{\Sd}{\mathcal{S}^\dagger}

\DeclareMathOperator{\V}{V} % Vertex set of a simplicial complex
\DeclareMathOperator{\ess}{ess} % essential intersections

\newcommand{\T}{\mathcal{T}}
\newcommand{\Td}{\mathcal{T}^\dagger}

\newcommand{\dSd}{d_\S^\dagger}
\newcommand{\dTd}{d_\T^\dagger}
\newcommand{\dS}{d_\S}
\newcommand{\dT}{d_\T}

\renewcommand{\tilde}{\widetilde}

\newcommand{\tautwo}{\ensuremath{\tau^{(2)}}}

\newcommand{\orientedsum}{\ensuremath{
		\mathchoice{\,\raisebox{-0.06em}{\includegraphics[scale=0.4]{orientedsumsymbol.pdf}}}
		{\kern 0.1em\raisebox{-0.06em}{\includegraphics[scale=0.4]{orientedsumsymbol.pdf}}\kern 0.1em}
		{\raisebox{-0.06em}{\includegraphics[scale=0.28]{orientedsumsymbol.pdf}}}
		{\raisebox{-0.06em}{\includegraphics[scale=0.28]{orientedsumsymbol.pdf}}}}
}

\makeatletter
\@namedef{subjclassname@2020}{%
  \textup{2020} Mathematics Subject Classification}
\makeatother

\begin{document}

\title{The complex of hypersurfaces in a homology class}

\author{Gerrit Herrmann}
\address{Fakult\"{a}t f\"{u}r Mathematik, Universit\"{a}t Regensburg, 93040 Regensburg, Germany}
\email{herrmann.gerrit@gmail.com}

\author{José Pedro Quintanilha}
\address{Fakultät für Mathematik, Universität Bielefeld, Postfach 100131, D-33501 Bielefeld, Germany}
\email{zepedro.quintanilha@gmail.com}

\date{\today.} %\copyright{\ G.~Herrmann, J.~P.~Quintanilha 2020}.}

\thanks{This work was supported by the CRC~1085 \emph{Higher Invariants}
  (Universit\"at Regensburg, funded by the DFG)}
  
%\subjclass[2020]{
%57K10, % Knot theory
%57K20, % 2-dimensional topology (including mapping class groups of surfaces, Teichmüller theory, curve complexes, etc.)
%57K31, % Invariants of 3-manifolds (also skein modules; character varieties)
%57R40, % Embeddings in differential topology
%57R52} % Isotopy in differential topology

%\keywords{codimension-$1$ homology, oriented surgery, Kakimizu complex, Thurston norm, differential topology}

\begin{abstract}
	For a compact oriented smooth $n$-manifold~$M$ and a codimension-$1$ homology class $\phi \in \h_{n-1}(M, \bd M)$, we investigate a simplicial complex~$\Sd(M, \phi)$ relating the properly embedded hypersurfaces in~$M$ representing~$\phi$. Its definition is akin to that of other classical complexes, such as the curve complex of a surface or the Kakimizu complex of a knot, with the difference that hypersurfaces are not taken up to isotopy.
	
	We prove that $\Sd(M, \phi)$~is connected and simply connected in every dimension~$n$. We also show connectedness of a similar complex $\Td(M, \phi)$ adapted to the $3$-dimensional case, where only Thurston norm-realizing surfaces are considered. The connectedness results are transported to the complexes $\S(M, \phi), \T(M, \phi)$ where hypersurfaces are taken up to isotopy, and for $n=2$ the simple connectedness result carries over as well. We also briefly discuss extensions to a context studied by Turaev, where regular graphs in $2$-complexes are used to represent $1$-dimensional cohomology classes.
	
	We finish with two applications: we give an alternative proof of the fact that all Seifert surfaces for a fixed knot in a rational homology sphere are tube-equivalent, and we use connectedness of~$\Td(M, \phi)$ to define a new $\ell^2$-invariant of $2$-dimensional homology classes in irreducible and boundary-irreducible oriented compact connected $3$-manifolds with empty or toroidal boundary.
\end{abstract}
\maketitle

\section{Introduction}

\subsection{Connecting homologous hypersurfaces}

Given a smooth manifold~$M$, by a \textbf{submanifold} of~$M$ we mean a smooth manifold~$S$ contained in~$M$, whose inclusion~$S \into M$ is a smooth embedding transverse to the boundary~$\partial M$. We will say that~$S$ is \textbf{properly embedded} if $S\cap \bd M = \bd S$. When its codimension $\dim M - \dim S$ equals~$1$, we call~$S$ a \textbf{hypersurface}.

For an oriented compact smooth manifold~$M$ of dimension~$n$ and a codimension-$1$ (integral) homology class~$\phi\in \h_{n-1}(M,\bd M)$,  it is well-known that there is a properly embedded oriented compact hypersurface~$S\subset M$ \textbf{representing}~$\phi$, that is,
for which $\phi$~is the image of the fundamental class~$[S]$ under the inclusion-induced map~$\h_{n-1}(S, \bd S) \to \h_{n-1}(M, \bd M)$. (One can find such an~$S$ as the pre-image of a regular value of a smooth map~$M \to \SS^1$ classifying the Poincaré dual of~$\phi$.) In this article, we relate the various such embedded hypersurfaces, the first main theorem being the following:

\begin{thm}[Sequentially disjoint hypersurfaces]\label{thm:maintheorem}
	Let $M$~be an oriented compact smooth manifold of dimension~$n$, let $\phi\in \h_{n-1}(M, \bd M)$ be a codimension-$1$ homology class, and let $S,S'$~be properly embedded oriented hypersurfaces in~$M$ representing~$\phi$. Then there is a sequence~$S=S_0, S_1,\ldots, S_m=S'$ of properly embedded oriented hypersurfaces in~$M$, all representing~$\phi$, such that each two consecutive~$S_i$ are disjoint.
\end{thm}

Our proof produces the intermediate hypersurfaces~$S_i$ rather explicitly.
Namely, it will be clear that they can all be chosen to lie in an arbitrarily small neighborhood of the union~$S \cup S'$, and if $S,S'$~have disjoint boundaries, then for every~$S_i$, each connected component of~$\bd S_i$ is isotopic in~$\bd M$ to a component of $\bd S$~or~$\bd S'$. % Moreover, if $S$~and~$T$ are transverse, one may take $S_0 = S$ and $S_m=T$.

Theorem~\ref{thm:maintheorem} will be re-stated and proved as Theorem~\ref{thm.connected}. It is rephrased in the language of a simplicial complex~$\Sd(M, \phi)$, whose vertices are the hypersurfaces representing~$\phi$, and whose $k$-simplices are sets of $k+1$ pairwise-disjoint hypersurfaces (Definition~\ref{dfn.Sdagger}). With that terminology in place, Theorem~\ref{thm:maintheorem} is the statement that $\Sd(M, \phi)$~is connected.

\subsection{Decomposing oriented surgeries}

The overall idea of the proof of Theorem~\ref{thm:maintheorem} is to first replace $S$~and~$S'$ with hypersurfaces intersecting transversely, and then to perform a surgery procedure that yields a hypersurface~$\Sigma$ representing the class~$2 \phi$. The surface~$\Sigma$ is then the disjoint union of two hypersurfaces~$T_0, T_1$, each representing~$\phi$, and we can observe that at least one of the $T_i$~must have strictly fewer intersections with both $S$~and~$S'$ than $S$~and~$S'$ have with one another. The theorem then follows by an inductive argument.

In fact, our proof provides a finer control on the number of components in the intersection of the~$T_i$ with $S$~and~$S'$. This yields a linear upper bound on the distance between two transverse hypersurfaces~$S,S'$ in~$\Sd(M, \phi)$, in terms of the number of components in~$S \cap S'$ (Proposition~\ref{prop.distbound}). 

These ideas are further developed into a more technically involved argument showing the second main result of this paper, which is re-stated and proved in the text as Theorem~\ref{thm.simplyconn}:

\begin{thm}[Simple connectedness of~$\Sd(M, \phi)$]\label{thm:maintheorem2}
	Let $M$~be an oriented compact smooth manifold of dimension~$n$ and let $\phi\in \h_{n-1}(M, \bd M)$ be a codimension-$1$ homology class. Then $\Sd(M, \phi)$~is simply connected.
\end{thm}

We emphasize that both Theorems \ref{thm:maintheorem}~and~\ref{thm:maintheorem2} are provided with dimension-agnostic proofs. We are not aware of similar earlier results in dimensions above~$3$.

\subsection{Other complexes of hypersurfaces}

Similar complexes of hypersurfaces have been studied before, but traditionally their vertices are embedded submanifolds \emph{up to isotopy}. A classical example is the curve complex~$\mathcal C(M)$ of a surface~$M$. The vertices of~$\mathcal C(M)$ are the isotopy classes of (unoriented) simple closed curves in~$M$, and a finite set of isotopy classes is a simplex whenever those isotopy classes can be represented by pairwise-disjoint curves. The curve complex was introduced by Harvey \cite{Har81} and has proved a fruitful tool in the study of surface mapping class groups;  see for example the exposition in the book of Farb and Margalit \cite[Section~4.1]{FM12}. By fixing a primitive homology class $x\in \h_1(S)$  and taking the subcomplex~$\mathcal{C}_x(M)$ of~$\mathcal{C}(M)$ spanned by the isotopy classes of curves that can be oriented to represent~$x$, Hatcher and Margalit have given a new proof of the fact that the Torelli group of~$M$ is generated by certain easy to understand mapping classes \cite[Theorem~1]{HM12}. A key step is showing that if $M$~has genus at least~$3$, the complex~$\mathcal{C}_x(M)$ is connected \cite[Theorem~2]{HM12}. Irmer established various geometric properties of a variant of this complex where the vertices are homotopy classes of multicurves in a fixed homology class \cite{Ir12}.

Moving to dimension~$3$, Kakimizu has studied complexes of isotopy classes of Seifert surfaces in a link exterior that are incompressible, or that have minimal genus. For non-split links, Kakimizu proved that such complexes are connected \cite[Theorem~A]{Ka92}. Przytycki and Schultens have adapted his complex of minimal Seifert surfaces in a knot exterior to a wider class of $3$-manifolds and proved contractibility of the complex in that setting \cite[Theorem~1.1]{PS12}.

More recently, Bowden, Hensel and Webb considered a variant~$\mathcal{C}^\dagger(M)$ of the curve complex where the vertices are actual curves, rather than isotopy classes. This allowed them to show that the space of unbounded quasi-morphisms on~$\mathrm{Diff}_0(M)$ (for an orientable closed surface~$M$ of positive genus) is infinite-dimensional \cite[Theorem~1.2]{BHW20}. We follow their usage of the ``dagger'' superscript in the notation for our complex~$\Sd(M, \phi)$, to emphasize that hypersurfaces are not taken up to isotopy.

%\note{A few more references: \cite{Schu17} \cite{Schu18}. Refer maybe to the discussion in \cite[Section 2]{Schu17}}

\subsection{Additional results}

When $M$~is assumed to have dimension~$3$ and to be irreducible and boundary-irreducible (Definition~\ref{dfn.irred}), it is particularly interesting to focus on surfaces that are ``efficient representatives'' of~$\phi$. Concretely, they are Thurston norm-realizing (Definition~\ref{dfn.thurstonnorm}). With that in mind, we will adjust the proof of Theorem~\ref{thm:maintheorem} to the subcomplex~$\Td(M,\phi)$ of~$\Sd(M,\phi)$ spanned by the vertices that realize the Thurston norm and have no homologically trivial parts (Definition~\ref{dfn.Tdagger}):

\begin{thm}[Connectedness of $\Td(M, \phi)$]\label{thm:3mfdcaseintro}
	Let $M$~be an irreducible and boundary-irreducible oriented compact smooth $3$-manifold, and let $\phi\in \h_2(M, \bd M)$. Then the complex~$\Td(M, \phi)$ is connected.
\end{thm}

This result is re-stated and proved as Theorem~\ref{thm.3mfdcase}.

We will also consider the simplicial complexes $\S(M,\phi)$~and~$\T(M,\phi)$ defined similarly to their ``dagger'' counterparts, but with all hypersurfaces taken up to smooth proper isotopy, and with a finite set of isotopy classes spanning a simplex if they can all be simultaneously disjointly realized (Definitions \ref{dfn.S}~and~\ref{dfn.T}).

\begin{thm}[Dropping the daggers]
	Let $M$~be an oriented compact smooth manifold of dimension~$n$, let $\phi\in \h_{n-1}(M, \bd M)$ be a codimension-$1$ homology class. Then:
	\begin{itemize}
		\item the complex~$\S(M,\phi)$ is connected (Corollary~\ref{cor.nodagger}),
		\item if $M$~is a reducible and boundary-irreducible $3$-manifold, then $\T(M, \phi)$~is connected (Corollary~\ref{cor.3mfdnodagger}),
		\item if $n=2$, then $\S(M,\phi)$ is simply connected (Corollary~\ref{cor.2mfdcase}).
	\end{itemize}
\end{thm}

The first two items will follow directly from previous results, with upper bounds on distance also inherited. The third one, however, requires an additional input available only in dimension~$2$ -- the Bigon Criterion (Theorem~\ref{thm.bigoncrit}). It is unclear to us whether it is possible to circumvent it and show simple connectedness of~$\S(M, \phi)$ in all dimensions. A sufficient condition would be an affirmative answer to Question~\ref{quest.MillionDollarQuest} below.

\subsection{Outline of the article}

In Section~\ref{sec.genpos}, we set up useful terminology for discussing collections of embedded submanifolds, and prove a lemma regarding general position.

Section~\ref{sec.conn} introduces the main character of this paper, the simplicial complex $\Sd(M, \phi)$ of a manifold~$M$ and a codimension-$1$ homology class~$\phi$ of~$M$. There, we prove the first main result, Theorem~\ref{thm.connected}. The core of the argument is Proposition~\ref{prop.distbound}, which gives a distance bound between vertices of~$\Sd(M, \phi)$ corresponding to transverse hypersurfaces.
In Section~\ref{sec.thurston}, we adjust these arguments to prove connectedness of the complex~$\Td(M, \phi)$ of Thurston norm-realizing surfaces in a reducible and boundary-irreducible oriented compact smooth $3$-manifold~$M$ (Theorem~\ref{thm.3mfdcase}).

The second main result, that $\Sd(M, \phi)$~is simply connected, is presented in Section~\ref{sec.simpcon}, as Theorem~\ref{thm.simplyconn}. The proof is similar in spirit to that of connectedness, but more technically involved.

In Section~\ref{sec.dropdagger}, we consider the complexes $\S(M, \phi)$~and~$\T(M, \phi)$ where all hypersurfaces are taken up to smooth proper isotopy. We explain how connectedness of $\Sd(M, \phi)$ and $\Td(M,\phi)$ implies connectedness of these smaller complexes, with similar upper bounds on distance (Corollaries \ref{cor.nodagger}~and~\ref{cor.3mfdnodagger}). We also transfer the simple connectedness result for~$\Sd(M, \phi)$ to $\S(M,\phi$), in the case where $M$~has dimension~$2$ (Corollary~\ref{cor.2mfdcase}).

In Section~\ref{sec.turaev}, we briefly comment on a generalization of the results in the $2$-dimensional case to the setting of graphs embedded in $2$-dimensional CW-complexes (under certain regularity assumptions). This connects to work of Turaev \cite{Tu02}, who has used such graphs to represent $1$-dimensional cohomology classes, and defined an analogue of the Thurston norm.

We finish by presenting a pair of applications of our results to $3$-manifold topology. Namely, in Section~\ref{sec.seifert} we give an alternative proof of the classical theorem that all Seifert surfaces for a knot in a rational homology $3$-sphere are tube-equivalent (Theorem~\ref{thm.tubeq}), and in Section~\ref{sec:l2torsion} we explain how connectendess of the complex~$\Td(M, \phi)$ has been used to construct an $\ell^2$-invariant for $2$-dimensional homology classes in irreducible and boundary-irreducible compact oriented connected smooth $3$-manifolds with empty or toroidal boundary (Corollary~\ref{cor.l2invariant}).

\subsection{Acknowledgments}
The authors are grateful to Stefan Friedl and Clara Löh for valuable feedback on earlier versions of the manuscript, to Vladimir Turaev for pointing out the connection to his work discussed in Section~\ref{sec.turaev}, to an anonymous referee who gave an extrmely constructive report on an earlier version of the text, and to Jonathan Glöckle for several illuminating conversations about transversality of smooth submanifolds.

\section[General position]{A note on general position}\label{sec.genpos}

The proofs of all main results in this article involve performing geometric constructions on families of submanifolds of a fixed smooth ambient manifold~$M$. These procedures can only be carried out if the submanifolds involved satisfy a ``general position'' assumption, which we explain in the current section. The reader who is uninterested in the technical details involved in perturbing manifolds into general position is invited to read only until the statement of Proposition~\ref{prop.transverse}, and then skip to the next section. 

Wall's book \cite[Section~1.5]{Wa16} is the main reference for all definitions in differential topology that are not explicitly stated here. We deviate from his terminology only in the use of the word ``\textbf{proper}'': for us, it always refers to the condition $S\cap \bd M = \bd S$ on a submanifold $S \subset M$, or, more generally, to the condition $f^{-1} (\partial M) = \partial N$ on a smooth map of manifolds $f\colon N \to M$. Since we will deal exclusively with compact spaces, the alternative notion of properness \cite[Section~A.2]{Wa16} will not be needed.

We now introduce a notion of transversality for finite sets of properly embedded submanifolds in an ambient manifold.

\begin{dfn}\label{dfn.transverse}
	Let $M$ be a smooth manifold. A finite set~$\mathcal U = \{S_1, \ldots, S_k\}$ of proper submanifolds of~$M$ is \textbf{transverse} if
	for every subset $I \subset \{1, \ldots,k\}$, the intersection~$\bigcap_{i\in I} S_i$ is a submanifold of~$M$, and for every pair of disjoint subsets~$I,J \sub \{1, \ldots, k\}$, the ubmanifolds $\bigcap_{i\in I} S_i$~and~$\bigcap_{j \in J} S_j$ are transverse.
\end{dfn}

Note that since submanifolds are by definition transverse to the boundary, Definition~\ref{dfn.transverse} requires in particular that all intersections~$\bigcap_{i\in I} S_i$ be transverse to~$\bd M$. It then follows also that $\bigcap_{i\in I} S_i$~is properly embedded: indeed, one can see in general that for any submanifold~$T$ of~$M$, the fact that $T$~is transverse to~$\bd M$ implies that $T \cap \bd M \subseteq \bd T$, so in particular $\big(\bigcap_{i\in I} S_i\big) \cap \bd M \subseteq \bd \big(\bigcap_{i\in I} S_i\big)$. Conversely, each point $p \in \bd \big(\bigcap_{i\in I} S_i\big)$ is in the boundary of some~$S_i$, and hence, since $S_i$~is properly embedded, we have $p\in \bd M$, and so~$p \in \big(\bigcap_{i\in I} S_i\big) \cap \bd M$.

Figure~\ref{fig.nottransverse} exemplifies two phenomena that Definition~\ref{dfn.transverse} excludes.

\begin{figure}[h]
	\centering
	\def \svgwidth{0.7\linewidth}
	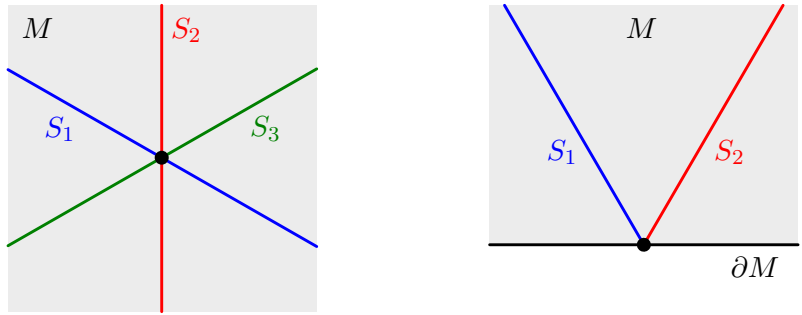
	\caption{Submanifolds of a surface~$M$ that do not form a transverse set. Left: the curves~$S_1, S_2, S_3$ meet pairwise-transversely at a point, but the intersection~$S_1 \cap S_2$ is not transverse to~$S_3$. Right: for two curves~$S_1, S_2$ meeting transversely at~$\bd M$, the intersection~$S_1 \cap S_2$ is not transverse to~$\bd M$, so it is not a submanifold of~$M$.}
	\label{fig.nottransverse}
\end{figure}

The goal of the current section is to establish the following statement, which justifies thinking of transverse sets as being in general position.

\begin{prop}[Transverse approximation]\label{prop.transverse}
	Suppose $\mathcal U$ is a transverse set of properly embedded submanifolds of a compact smooth manifold~$M$, and let~$f\colon T \into M$ be a proper embedding of a compact manifold~$T$. Then $f$~can be perturbed by an arbitrarily small proper isotopy to a proper embedding~$g\colon T \into M$, such that for the modified manifold~~$T':=g(T)$, the set~$\mathcal U \cup  \{T'\}$ is transverse.
\end{prop}

Here the phrase ``arbitrarily small isotopy'' warrants some explanation. Given two smooth manifolds $T, M$, the set~$\C^\infty (T,M)$ of smooth maps $T\to M$ is typically endowed with either the $C^\infty$~topology or the $W^\infty$~topology, which are the same if $T$~is compact \cite[Appendix A.4]{Wa16}; we will thus no longer care to distinguish them. If we consider the subspace~$\C_\bd^\infty (T,M)$ of proper maps, we can make the statement of Proposition~\ref{prop.transverse} precise by expressing it in terms of this topology. This translation relies on the following result.

\begin{prop}[Stability of proper embeddings]\label{prop.embnbh}
	If $T,M$ are smooth manifolds, with $T$~compact, and $f\colon T\into M$~is a proper embedding, then there is a neighborhood~$U$ of~$f$ in~$\C_\bd^\infty (T,M)$ such that every~$g\in U$ is a proper embedding that is properly isotopic to~$f$.
\end{prop}

\begin{proof}[Proof sketch]
	If we do not insist that the isotopy connecting $f$~and~$g$ be proper, then this statement is proved in Wall's book \cite[Proposition~4.4.4]{Wa16}. But the stronger result actually follows from the same argument, with almost no modification. Indeed, that proof uses a map~$H\colon W \times [0,1] \to M$, where $W$~is an appropriate neighborhood of the diagonal in~$M\times M$. This map~$H$ is constructed by putting a Riemannian metric on~$M$ and using the existence of unique geodesics between pairs of points that are close enough.
	
	But if one starts with a Riemannian metric for which $\bd M$~is totally geodesic (which we can do \cite[Proposition~2.3.7~(i)]{Wa16}), then geodesics connecting boundary points are contained in the boundary, and this fact translates into properness of the isotopy that is ultimately produced between $f$~and~$g$.
\end{proof}

One can therefore refine Proposition~\ref{prop.transverse} as follows.

\begin{prop}[Denseness of transverse proper embeddings]\label{prop.transverse'}
	Let $\mathcal U$~be a transverse set of properly embedded submanifolds of a compact smooth manifold~$M$, and let $T$~be a compact smooth manifold. Then the set of proper embeddings~$f\colon T \into M$ making $\mathcal U \cup \{f(T)\}$~transverse
	%\[\{f\in \C_\bd^\infty (T, M) \st \text{$f$ is a proper embedding and $\mathcal U \cup \{f(T)\}$~is transverse}\} \]
	is dense in the (open) subset of~$\C_\bd^\infty (T, M)$ consisting of proper embeddings.
\end{prop}

The main tool one uses in order to establish statements of this type is Thom's Transversality Theorem \cite[Theorem~4.5.6]{Wa16}. We will not need its full power, only the following corollary.

\begin{thm}[Elementary Transversality Theorem]\label{thm.elementtransv}
	Let $T,M$ be smooth manifolds, with $T$~compact, and let $S$~be a closed submanifold of~$M$. Then the set of maps~$f \colon T \to M$~transverse to~$S$ is open and dense in~$\C^\infty (T, M)$.
	
	Suppose further that $f_0 \colon T \to M$~is a smooth map such that the restriction $f_0|_{\bd T}$ is transverse to~$S$, and consider the subspace $\C^\infty (T, M; f_0, \bd T) \sub \C^\infty (T, M)$ of maps whose restriction to~$\bd T$ agrees with~$f_0$. Then the set of maps $f \in \C^\infty (T, M; f_0, \bd T)$ transverse to~$S$ is open and dense in~$\C^\infty (T, M; f_0, \bd T)$.
\end{thm}

The proof of the first part of the Elementary Transversality Theorem can be found in the book by Golubitsky and Guillemin \cite[Corollary~4.12]{GG73}, and the second statement follows from a stronger version of Thom's Transversality Theorem \cite[Proposition~4.5.7]{Wa16}, using the same argument.

%This result is typically expressed in the language of spaces of jets~$\mathrm J^r(T, M)$ (with $r\in \NN$), but as we don't need this much generality, we state only the special case $r=0$, where the space of jets is simply the product~$T\times M$.

%\begin{thm}[Thom's Transversality Theorem]
%Let $T,M$ be smooth submanifolds, with $T$~compact, and let $N$~be a closed submanifold of~$T\times M$. Then the set of maps~$f \colon T \to M$ with $(\id, f) \colon T \to T \times M$~transverse to~$N$ is open and dense in~$\C^\infty (T, M)$.
%Moreover, suppose $f_0 \colon T \to M$~is a smooth map such that the restriction $(\id, f_0)|_{\bd T}\colon \bd T \to T\times M$ is transverse to~$N$, and consider the subspace $\C^\infty (T, M; f_0, \bd T) \sub \C^\infty (T, M)$ of maps whose restriction to~$\bd T$ agrees with~$f_0$. Then the set of maps $f \in \C^\infty (T, M; f_0, \bd T)$ with $(\id ,f) \colon T \to T\times M$~transverse to~$N$ is open and dense in~$\C^\infty (T, M; f_0, \bd T)$.
%\end{thm}

Before proving Proposition~\ref{prop.transverse'}, we state and prove two lemmas, the first of which is a mere linear-algebraic observation.

\begin{lem}[3-fold transversality]\label{lem.linalgtransversality}
	Let $V$~be a finite-dimensional vector space (over any field), and let $T,S,R$~be pairwise transverse subspaces of~$V$, that is, $T+S=T+R =S+R = V$. Then the following conditions are equivalent:
	\begin{itemize}
		\item $T + (S \cap R) = V$,
		\item $S + (T \cap R) = V$,
		\item $R + (T \cap S) = V$.
	\end{itemize}
\end{lem}
\begin{proof}
	A straightforward dimension count shows that each condition is equivalent to
	\[ \dim T + \dim S + \dim R - \dim (T \cap S \cap R) =  2 \dim V. \qedhere\]
\end{proof}

\begin{lem}[Transversality criterion]\label{lem.transversalityreduction}
	Let $\mathcal{U}:=\{S_1, \ldots, S_k\}$ be a transverse set of properly embedded submanifolds of a compact smooth manifold~$M$, and let $T$~be a properly embedded submanifold of~$M$ such that for every non-empty subset~$I \sub \{1, \ldots, k\}$ the following conditions hold:
	\begin{itemize}
		\item $T$ is transverse to~$\bigcap_{i\in I} S_i$, and
		\item $\bd T$ is transverse to~$\bigcap_{i \in I} \bd S_i$ in~$\bd M$.
	\end{itemize}
	Then $\mathcal U \cup \{T\}$~is a transverse set.
\end{lem}

\begin{proof}
	Two conditions need to be verified, for all disjoint subsets $I, J \sub \{1, \ldots, k\}$:
	\begin{enumerate}
		\item the intersection~$T \cap \bigcap_{i \in I}S_i$ is a submanifold of~$M$,
		\item the submanifold~$T \cap \bigcap_{i \in I}S_i$ is transverse to~$\bigcap_{j \in J}S_j$.
	\end{enumerate}
	
	For proving (1), the fact that $T$~is transverse to~$\bigcap_{i\in I} S_i$ tells us that $T \cap \bigcap_{i \in I}S_i$~is a manifold embedded in~$T$ \cite[Lemma~4.5.1]{Wa16}, and hence in~$M$.
	We are thus left to show that $T \cap \bigcap_{i \in I}S_i$~is transverse to~$\bd M$.
	
	Since $T$~is transverse to~$\bd M$, the tangent space~$\TT_p (T)$ at each boundary point $p\in \bd T$ has a $1$-dimensional subspace~$R$ such that
	\[\TT_p(T) = \TT_p(\bd T) \oplus R,\ \qquad	\TT_p(M) = \TT_p(\bd M) \oplus R.\]
	Assuming now that $p$~is in~$\bd M \cap T \cap \bigcap_{i\in I} S_i$, we obtain
	\begin{align*}
		\TT_p(M) &= \TT_p(\bd M) + R  &\text{(second equality)}\\
		&= \TT_p\biggl(\bigcap_{i \in I}\bd S_i\biggr) + \TT_p(\bd T) +  R &\textstyle \text{($\bd T$ transverse to~$\bigcap_{i \in I}\bd S_i$ in $\bd M$)}\\
		&= \TT_p\biggl(\bigcap_{i \in I}\bd S_i\biggr) + \TT_p(T)  &\text{(first equality)}\\
		&= \biggl(\TT_p\biggl(\bigcap_{i \in I} S_i\biggr) \cap \TT_p(\bd M)\biggr) + \TT_p(T)  & \text{(all $S_i$ properly embedded)}\\
		&= \biggl(\TT_p\biggl(\bigcap_{i \in I} S_i\biggr) \cap \TT_p(T)\biggr) + \TT_p(\bd M) & \textstyle \text{($T$~transverse to~$\bigcap_{i\in I}S_i$, Lemma~\ref{lem.linalgtransversality})}\\
		&= \biggl(\TT_p\biggl(T\cap \bigcap_{i \in I} S_i\biggr)\biggr) + \TT_p(\bd M).&
	\end{align*}
	Therefore, $T \cap \bigcap_{i \in I}S_i$~is transverse to~$\bd M$.
	
	Condition (2) follows from a straightforward application of Lemma~\ref{lem.linalgtransversality} to the tangent spaces of $T$, $\bigcap_{i\in I}S_i$ and $\bigcap_{j\in J}S_j$, at points where all these submanifolds meet.
\end{proof}

Finally, we tackle the main result of this section.

\begin{proof}[Proof of Proposition~\ref{prop.transverse'}]
	We will show that every proper submanifold~$T \sub M$ (or, to be more precise, its inclusion $\iota \colon T \into M$) can be approximated arbitrarily well by a proper embedding~$f\colon T \into M$ for which $f(T)$~satisfies the conditions in Lemma~\ref{lem.transversalityreduction}.
	
	Applying the first part of the Elementary Transversality Theorem, we see that for each~$I\sub \{1, \ldots, k\}$, the set of embeddings $\bd T \into \bd M$ transverse to~$\bd \bigl(\bigcap_{i \in I}S_i\bigr)$ in~$\bd M$ is open and dense in~$\C^\infty(\bd T, \bd M)$. Hence, so is the set of embeddings simultaneously satisfying this transversality condition for all (finitely many) subsets~$I$. We can thus approximate the restriction~$\iota|_{\bd T}$ arbitrarily well by a map~$f_\bd\colon \bd T \into \bd M$ transverse to all~$\bd\bigl(\bigcap_{i \in I}S_i\bigr)$. By Proposition~\ref{prop.embnbh}, we may take $f_\bd$~to be an embedding.
	
	One can now use a small isotopy from~$\iota|_{\bd T}$ to~$f_\bd$ in order to approximate~$\iota$ by a proper embedding~$f_0 \colon T \to M$ that differs from~$\iota$ by a small proper isotopy supported in a collar neighborhood of~$\bd M$, and such that $f_0|_{\bd T} = f_\bd$ and $f_0$~is transverse to~$\bd M$. This construction relies on the existence of tubular neighborhoods for submanifolds with boundary \cite[Theorem~2.3.8]{Wa16}.
	
	Finally, we note that for each~$I \sub \{1, \ldots, k\}$, the fact that $f_\bd$~is transverse to~$\bd \bigl(\bigcap_{i\in I} S_i\bigr)$ in~$\bd M$ implies that $f_\bd$~is transverse to~$\bigcap_{i \in I} S_i$ in~$M$, and so we can apply the second part of the Elementary Transversality Theorem to conclude that the set of maps in~$\C^\infty(T, M; f_0, \bd T)$ that are transverse to~$\bigcap_{i \in I} S_i$ is open and dense. Thus, as before, the set of maps satisfying this transversality condition for all subsets~$I$ is also dense, and so we can approximate~$f_0$ arbitrarily well by such a map~$f$. Again by Proposition~\ref{prop.embnbh} we can take $f$ to be a proper embedding. The submanifold $f(T)$~satisfies the conditions in Lemma~\ref{lem.transversalityreduction}, so we are done.
\end{proof}

\section[$\Sd(M, \phi)$ is connected.]{The complex $\Sd(M, \phi)$ is connected.}\label{sec.conn}

Throughout this article, we will make heavy usage of the (combinatorial) notion of a simplicial complex, so we briefly recall the basic definitions.

\begin{dfn} A \textbf{simplicial complex}~$S$ is the data of a set~$\V(S)$, called the \textbf{vertex set} of~$S$, and a subset of the power set of~$\V(S)$, whose elements are called \textbf{simplices}, such that:
	\begin{itemize}
		\item Each simplex~$\sigma$ is non-empty and finite. If $\sigma$~has $k+1$~elements, we say $\sigma$~has \textbf{dimension}~$k$, or that it is a \textbf{$k$-simplex}. Simplices of dimension~$1$ and~$2$ will sometimes be called \textbf{edges} and \textbf{triangles}, respectively.
		\item Every non-empty subset of a simplex~$\sigma$ is also a simplex, and is said to be a \textbf{face} of~$\sigma$.
	\end{itemize}
	
	The (possibly infinite) supremum among the dimensions of all simplices in~$S$ is the \textbf{dimension} of~$S$.
	
	A \textbf{map of simplicial complexes} $T\to S$ is a function $\V(T) \to \V(S)$ taking simplices of~$T$ to simplices of~$S$ (although not necessarily preserving their dimensions). A \textbf{subcomplex}~$T$ of~$S$ is a simplicial complex with $\V(T) \subseteq \V(S)$ and whose simplices are also simplices of~$S$. In particular, for each $k \in \NN$, the \textbf{$k$-skeleton} of~$S$  is the subcomplex of~$S$ with the same vertex set as~$S$, but only the simplices of dimension at most~$k$. 
	
	Every simplicial complex~$S$ gives rise to a topological space~$|S|$, its \textbf{geometric realization},
	constructed as follows: for each simplex~$\sigma$ of~$S$, let $\Delta_\sigma$ be a copy of the standard simplex of dimension~$\dim(\sigma)$ in~$\RR^{\dim(\sigma)+1}$, with the vertices of $\Delta_\sigma$~labeled by the vertices of~$\sigma$. Then take
	\[|S| := \bigg( \bigsqcup_{\text{$\sigma$ simplex of $S$}} \Delta_\sigma \bigg) / {\sim},\]
	where $\sim$~is generated by the affine maps $\Delta_\tau \into \Delta_\sigma$ given on vertices by the inclusions $\tau \into \sigma$ whenever $\tau$~is a face of~$\sigma$.
	This is a functorial construction: every map of simplicial complexes~$S\to T$ induces a continuous map $|S| \to |T|$ by extending the assignment on vertices $\V(S) \to \V(T)$ to affine maps on the~$\Delta_\sigma$ for all simplices~$\sigma$ of~$S$.
\end{dfn}

We will often blur the distinction between a simplicial complex~$S$ and its geometric realization, writing statements like ``$S$~is connected'' when referring to topological features of~$|S|$. Also, when showing connectedness of a simplicial complex, it suffices to prove it for the $1$-skeleton, which is a graph. In that case, we will always employ the equivalent combinatorial notion of connectedness for graphs, as the existence of a sequence of adjacent edges between any two vertices.

This article is devoted to the study of the following simplicial complex.

\begin{dfn}\label{dfn.Sdagger}
	Let $M$ be an oriented smooth manifold of dimension~$n$, and $\phi\in \h_{n-1}(M,\bd M)$~a codimension-$1$ homology class. We denote by~$\Sd(M, \phi)$ the simplicial complex defined as follows:
	\begin{itemize}
		\item The vertices are the (possibly disconnected) properly embedded oriented smooth hypersurfaces in~$M$ representing~$\phi$.
		\item A set of $k+1$ hypersurfaces as above is a $k$-simplex if those hypersurfaces are pairwise-disjoint.
	\end{itemize}
\end{dfn}

The first main result of this paper, which we shall prove in this section, is the following.

\begin{thm}[Connectedness of~$\Sd(M, \phi)$]\label{thm.connected}
	Let $M$~be an oriented compact smooth $n$-manifold, and $\phi\in \h_{n-1}(M,\bd M)$~a codimension-$1$ homology class. Then $\Sd(M, \phi)$~is connected.
\end{thm}

The codimension hypothesis is essential. For example, every non-trivial element~$\phi$ of~$\h_2 (\CC \mathrm P^2)$ has non-zero algebraic self-intersection, so the analogously defined complex~$\Sd(\CC \mathrm P^2, \phi)$ has no edges, although there are clearly infinitely many vertices. Similarly, the assumption that $M$~is orientable cannot be dropped, as one sees by taking $M=\RR\mathrm P^2$. In this case, the generator of $\h_1(\RR\mathrm P^2) \iso \ZZ/2$, when reduced to $\ZZ /2$-coefficients, has non-trivial algebraic self-intersection, and this again obstructs the existence of edges in the simplicial complex.

Throughout this section,  $M$ will always be an oriented compact smooth manifold of dimension~$n$, and $\phi\in \h_{n-1}(M,\bd M)$~a codimension-$1$ homology class. We will also denote by~$|X|$ the number of connected components of a topological space~$X$ and by~$\dSd(S_0,S_1)$ the path length distance between hypersurfaces $S_0$~and~$S_1$ in the $1$-skeleton of~$\Sd(M, \phi)$.

Theorem~\ref{thm.connected} will follow from the following statement, where we focus only on connecting \textit{transverse} vertices (see Definition~\ref{dfn.transverse}).

\begin{prop}[Distance bound on transverse hypersurfaces]\label{prop.distbound}
	Let  $S_0,S_1 \subset M$ be a transverse pair of hypersurfaces representing~$\phi$. If $k$~is a non-negative integer such that $|S_0 \cap S_1|<2^k$, then $\dSd(S_0,S_1) \le 2^k$.
\end{prop}

In particular, if $S_0 \cap S_1 \ne \emptyset$, then by choosing $k$~to satisfy $2^{k-1} \le |S_0 \cap S_1| <2^k$ we obtain the coarser (but easier to remember) estimate \[\dSd(S_0,S_1) \le 2^k = 2 \cdot 2^{k-1} \le 2 \, |S_0 \cap S_1|.\]

\begin{proof}%[Proof of Proposition~\ref{prop.distbound}]
	We proceed by induction over~$k$, and we will abbreviate $C:=S_0 \cap S_1$. If $k=0$, then we have~$|C|=0$, which means the~$S_i$ are disjoint. Hence they are connected by an edge and $\dSd(S_0,S_1)=1\le 2^0$.
	
	For positive~$k$, we will find a hypersurface~$T$ such that
	\begin{itemize}
		\item $T$ represents~$\phi$,
		\item $T$ is transverse to both~$S_i$, and
		\item $T$ has controlled intersection with each~$S_i$, in the following sense: \[|T \cap S_i| \le \frac{|C|}2 < 2^{k-1}.\]
	\end{itemize}
	By induction, it will follow that $\dSd(T, S_i) \le 2^{k-1}$ for each~$i$, and hence from the triangle inequality $\dSd(S_0, S_1) \le 2^k$.
	
	The overall strategy for finding~$T$ is to perform a certain surgery procedure on~$S_0, S_1$ in order to produce a third hypersurface~$\Sigma\subset M$ representing the homology class~$2 \phi$.  We then show that $\Sigma$~is the disjoint union of two hypersurfaces~$T_0, T_1$, each representing~$\phi$. Moreover, the set~$\{S_0, S_1, T_0, T_1\}$ is transverse, and we will observe that at least one of the~$T_m$ satisfies
	\[|T_m \cap S_i| \le \frac{|C|}2 \quad \text{for each~$i\in \{0,1\}$}\]
	(all these intersections are compact submanifolds of~$M$, and hence have finitely many components). This~$T_m$ will be our desired~$T$.
	
	To construct~$\Sigma$, we begin by observing that the normal bundle of the codimension-$2$ submanifold~$C$ of~$M$ is trivial. Indeed, since $S_0, S_1$~are both oriented, the orientation of~$M$ induces framings of~$S_0, S_1$, which jointly provide a framing of~$C$. Hence, there is an open neighborhood~$U$ of~$C$ in~$M$ that is diffeomorphic to~$C \times \RR^2$ via a diffeomorphism that identifies ${S_0 \cap U}$~with~$C \times \RR \times 0$, and $S_1 \cap U$~with~$C \times 0  \times \RR$, all respecting orientations.
	
	We construct~$\Sigma$ as follows (see Figure \ref{fig.orientedsum}):
	\begin{enumerate}
		\item Start with the union~$S_0 \cup S_1$.
		
		\item Replace a small neighborhood of~$C$ in~$S_0\cup S_1$ by a pair of smooth ramps connecting each side of~$C$ in~$S_0$ to a side of~$C$ in~$S_1$, in such a way that the resulting hypersurface inherits a consistent orientation from the~$S_i$. We make this construction more precise in the following paragraph, but the idea should be clear from the top right of Figure~\ref{fig.orientedsum}. 
		
		Consider the bump function~$\mathrm{Bp} \colon \RR \to \RR$ defined in Wall's book \cite[Section~1.1]{Wa16}, which satisfies
		\[ \begin{cases}
			\mathrm{Bp}(t) = 0 & \text{for $t\le 0$,}\\
			\mathrm{Bp}(t) \in {]0,1[}  & \text{for $0 < t < 1$, \quad and}\\
			\mathrm{Bp}(t) = 1  & \text{for $t\ge 1$.}
		\end{cases}\]
		We replace~$(S_0 \cup S_1) \cap U$ by the hypersurface corresponding, in~$C\times \RR^2$, to $C\times R$, where $R\subset \RR^2$~is the union of the two curves parameterized by \label{rampnotation}
		\begin{align*}
			t \mapsto \mathrm{Bp}(t)(t,0) + (1-\mathrm{Bp}(t))(0, t),\\
			t \mapsto \mathrm{Bp}(t)(0,t) + (1-\mathrm{Bp}(t))(t,0),
		\end{align*}
		with~$t\in \RR$.
		
		Note that the resulting hypersurface represents the homology class~$2\phi$, since the region~$C \times K$ of~$M$, where
		\[ K:=\{ t X \in \RR^2 \st t\in[0,1], X\in R \cap \DD^2 \}\quad \text{(suitably oriented),}\]
		exhibits the new hypersurface as homologous to~$[S_0] + [S_1]$. 
		
		\item Push this hypersurface slightly along its framing, so it intersects $S_0$~and~$S_1$ transversely, along a pair of copies of~$C$.
	\end{enumerate}
	We will say that any hypersurface~$\Sigma$ constructed in this manner is an \textbf{oriented surgery} of $S_0$~and~$S_1$.
	
	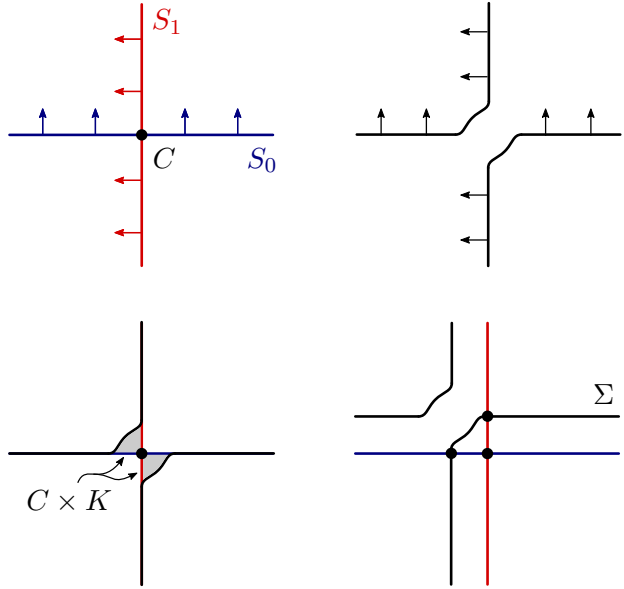
\begin{figure}[h]
		\centering
		\def \svgwidth{0.6\linewidth}
		%% Creator: Inkscape 1.0.2 (1.0.2+r75+1), www.inkscape.org
%% PDF/EPS/PS + LaTeX output extension by Johan Engelen, 2010
%% Accompanies image file 'orientedsum.pdf' (pdf, eps, ps)
%%
%% To include the image in your LaTeX document, write
%%   \input{<filename>.pdf_tex}
%%  instead of
%%   \includegraphics{<filename>.pdf}
%% To scale the image, write
%%   \def\svgwidth{<desired width>}
%%   \input{<filename>.pdf_tex}
%%  instead of
%%   \includegraphics[width=<desired width>]{<filename>.pdf}
%%
%% Images with a different path to the parent latex file can
%% be accessed with the `import' package (which may need to be
%% installed) using
%%   \usepackage{import}
%% in the preamble, and then including the image with
%%   \import{<path to file>}{<filename>.pdf_tex}
%% Alternatively, one can specify
%%   \graphicspath{{<path to file>/}}
%% 
%% For more information, please see info/svg-inkscape on CTAN:
%%   http://tug.ctan.org/tex-archive/info/svg-inkscape
%%
\begingroup%
  \makeatletter%
  \providecommand\color[2][]{%
    \errmessage{(Inkscape) Color is used for the text in Inkscape, but the package 'color.sty' is not loaded}%
    \renewcommand\color[2][]{}%
  }%
  \providecommand\transparent[1]{%
    \errmessage{(Inkscape) Transparency is used (non-zero) for the text in Inkscape, but the package 'transparent.sty' is not loaded}%
    \renewcommand\transparent[1]{}%
  }%
  \providecommand\rotatebox[2]{#2}%
  \newcommand*\fsize{\dimexpr\f@size pt\relax}%
  \newcommand*\lineheight[1]{\fontsize{\fsize}{#1\fsize}\selectfont}%
  \ifx\svgwidth\undefined%
    \setlength{\unitlength}{412.73454801bp}%
    \ifx\svgscale\undefined%
      \relax%
    \else%
      \setlength{\unitlength}{\unitlength * \real{\svgscale}}%
    \fi%
  \else%
    \setlength{\unitlength}{\svgwidth}%
  \fi%
  \global\let\svgwidth\undefined%
  \global\let\svgscale\undefined%
  \makeatother%
  \begin{picture}(1,0.86625371)%
    \lineheight{1}%
    \setlength\tabcolsep{0pt}%
    \put(0,0){\includegraphics[width=\unitlength,page=1]{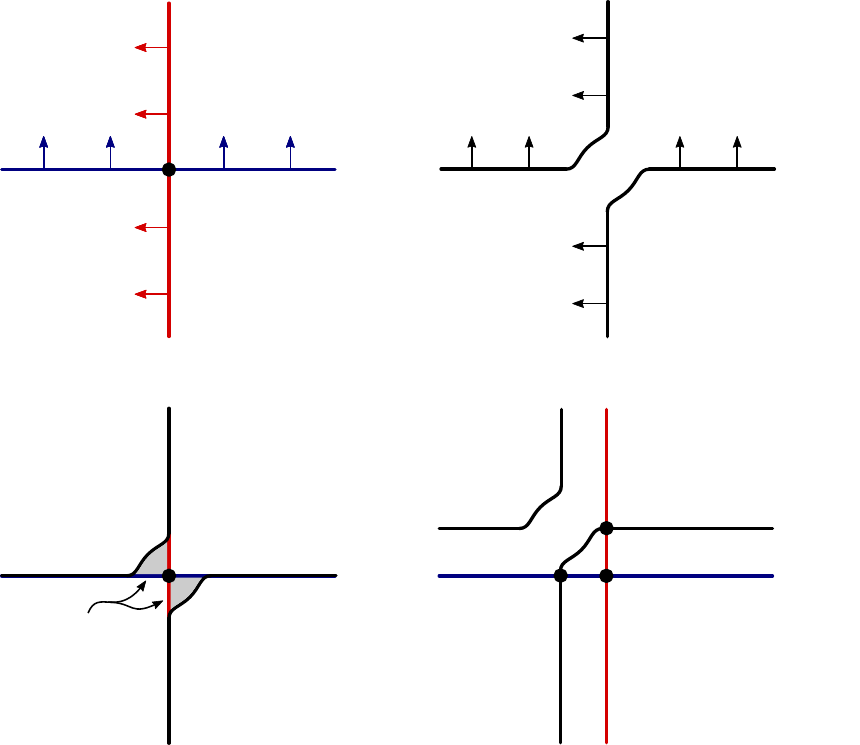}}%
    \put(3.57323592,-1.89394812){\color[rgb]{0,0,0}\makebox(0,0)[lt]{\lineheight{1.25}\smash{\begin{tabular}[t]{l}$\tilde P$\end{tabular}}}}%
    \put(4.12989236,-2.04137612){\color[rgb]{0,0,0}\makebox(0,0)[lt]{\lineheight{1.25}\smash{\begin{tabular}[t]{l}$w$\end{tabular}}}}%
    \put(0.86079722,0.26990931){\color[rgb]{0,0,0}\makebox(0,0)[lt]{\lineheight{1.25}\smash{\begin{tabular}[t]{l}$\Sigma$\end{tabular}}}}%
    \put(0.21123068,0.82569621){\color[rgb]{0.83137255,0,0}\makebox(0,0)[lt]{\lineheight{1.25}\smash{\begin{tabular}[t]{l}$S_1$\end{tabular}}}}%
    \put(0.35118851,0.6204962){\color[rgb]{0,0,0.50196078}\makebox(0,0)[lt]{\lineheight{1.25}\smash{\begin{tabular}[t]{l}$S_0$\end{tabular}}}}%
    \put(0.21279468,0.61971982){\color[rgb]{0,0,0}\makebox(0,0)[lt]{\lineheight{1.25}\smash{\begin{tabular}[t]{l}$C$\end{tabular}}}}%
    \put(0.02691094,0.11286208){\color[rgb]{0,0,0}\makebox(0,0)[lt]{\lineheight{1.25}\smash{\begin{tabular}[t]{l}$C \times K$\end{tabular}}}}%
  \end{picture}%
\endgroup%

		\caption{Performing oriented surgery on $S_0$~and~$S_1$. Top left: the local picture of~$S_0 \cup S_1$ in a neighborhood of~$C$, with framings of the~$S_i$ indicated by arrows. Top right: replacing a small neighborhood of~$C$ with a pair of ramps. The induced framing on the new hypersurface is illustrated. Bottom left: The shaded region corresponding to~$C \times K$ shows that the new hypersurface represents the class~$[S_0]+[S_1]$. Bottom right: properly isotoping this hypersurface along its framing yields the oriented surgery~$\Sigma$.}
		\label{fig.orientedsum}
	\end{figure}
	
	Our next goal is to show that $\Sigma$~is the disjoint union of two (possibly disconnected) hypersurfaces~$T_0, T_1$, each representing~$\phi$.
	% that are homologous to one another. Since $\h_{n-1}(M, \bd M)$~is torsion-free (a standard fact that follows from an application of Poincaré Duality and the Universal Coefficient Theorem for homology), this will allow us to deduce that $T_0$~and~$T_1$ both represent~$\phi$.
	We will say that the hypersurfaces~$T_0, T_1$ are obtained by \textbf{decomposing the oriented surgery~$\Sigma$} of~$S_0, S_1$. Let $f \colon M \to \SS^1$ be a continuous map with $f^{-1}(1) = \Sigma$  (such a map can be constructed by collapsing to a point the complement of an open tubular neighborhood of~$\Sigma$ in~$M$). Regarding~$\phi$ as an element of~$\h^1(M)$ and $f$~as a classifying map for~$2\phi$, we see that (for any basepoint) the induced map $f_* \colon \pi_1(M) \to \ZZ$ factors through $2\ZZ \into \ZZ$, and hence $f$~lifts to the double covering $\SS^1 \to \SS^1$, $z \mapsto z^2$. If $g \colon M \to \SS^1$~is the lifted map, then taking $T_0 := g^{-1}(1)$ and $T_1 := g^{-1}(-1)$ yields $\Sigma = T_0 \sqcup T_1$. Moreover, since $g$~is a classifying map for~$\phi$, we conclude $T_0$~and~$T_1$ both represent~$\phi$.

	All that is left is to see that at least one among~$T_0, T_1$ satisfies the claimed control on intersection with both~$S_i$. In fact, we will show more: there are non-negative integers~$N_0,N_1$ with $|C| = N_0 + N_1$, such that, for each~$m\in\{0,1\}$,
	\[|T_m \cap S_0| = |T_m \cap S_1| = N_m.\]
	Intuitively, ``the components of~$C$ are distributed among the~$T_m$''. In particular, for some~$m$ we have $N_m \le \frac{|C|}2$ and thus $T_m$~satisfies our claim.
	
	The existence of~$N_0, N_1$ is a consequence of the observation, plainly on display on the bottom right of Figure~\ref{fig.orientedsum}, that each component~$c$ of~$C$ gives rise to either:
	\begin{itemize}
		\item one component in each of the~$T_0\cap S_i$ and no component in either of the~$T_1 \cap S_i$ (if the sheet of~$\Sigma$ on the bottom right belongs to~$T_0$), or
		\item one component in each of the~$T_1\cap S_i$ and no component in either of the~$T_0 \cap S_i$ (if the sheet of~$\Sigma$ on the bottom right belongs to~$T_1$).
	\end{itemize}
	Each~$N_m$ is then the number of components of~$C$ that contribute to one (hence both) of the~$T_m \cap S_i$.
\end{proof}

\begin{proof}[Proof of Theorem~\ref{thm.connected}]
	Let $S_0, S_1$ be a (not necessarily transverse) pair of hypersurfaces representing~$\phi$. We produce a new properly embedded hypersurface~$S_1'$ by pushing-off~$S_1$ along its normal bundle, and perturbing it slightly to make it transverse to~$S_0$, while keeping it disjoint from~$S_1$ (this uses Proposition~\ref{prop.transverse}). Now apply Proposition~\ref{prop.distbound} to $S_0$~and~$S_1'$ to conclude that
	\[\dSd(S_0, S_1) \le \dSd(S_0, S_1') + \dSd(S_1', S_1) \le 2\,|S_0 \cap S_1'|+1,\]
	which proves the theorem.
\end{proof}

\section[Thurston norm-realizing surfaces]{Thurston norm-realizing surfaces in 3-manifolds}\label{sec.thurston}

We now study a variation of the simplicial complex from the previous section, where we consider only certain surfaces representing $2$-homology classes in irreducible and boundary-irreducible oriented compact smooth $3$-manifolds (see Definition~\ref{dfn.irred} below). These surfaces are, in a sense, most efficient: they realize the Thurston norm and have no homologically trivial parts (Definition~\ref{dfn.thurstonnorm}). Our goal is to show that restricting the complex from the previous section to the Thurston norm-realizing surfaces for a homology class still results in a connected complex. This will be accomplished simply by adjusting the proof of Proposition~\ref{prop.distbound}.

We begin by recalling some standard terminology.

\begin{dfn}\label{dfn.irred}
	Let $S$ be a compact smooth surface.
	\begin{itemize}
		\item A properly embedded circle in~$S$ is \textbf{inessential} if it bounds an embedded disc in~$S$. Otherwise, it is \textbf{essential}.
		
		\item A properly embedded arc in~$S$ is \textbf{inessential} if, together with an arc in~$\bd S$, it bounds an embedded disc in~$S$. Otherwise, it is \textbf{inessential}.
	\end{itemize}
	
	Let $M$~be a compact smooth $3$-manifold.
	\begin{itemize}
		\item $M$ is \textbf{irreducible} if every embedded $2$-sphere in~$M$ bounds an embedded $3$-ball.
		
		\item An embedded circle in~$\bd M$ is called a \textbf{meridian} if it is essential in~$\bd M$ but bounds a properly embedded disc in~$M$.
		
		\item $M$ is said to be \textbf{boundary-irreducible} if it contains no meridians.
	\end{itemize}
	
	Let $S$~be a properly embedded compact surface in~$M$.
	
	\begin{itemize}
		\item A \textbf{compressing disc} for~$S$ is a disc~$D$ embedded in~$M$ as a submanifold, with interior disjoint from~$S$, and whose boundary is either:
		\begin{itemize}
			\item an essential circle in~$S$, or
			\item the union of an essential arc in~$S$ and an embedded arc in~$\bd M$ (in which case $D$ is a submanifold with corner).
		\end{itemize} 
		We also demand that~$D$ intersect~$S$ transversely.
		\item If $S$~has a compressing disc, then $S$~is called \textbf{compressible}; otherwise it is \textbf{incompressible}.
	\end{itemize}
\end{dfn}

Note that if $S \subset M$~as above is a sphere or a disc, then $S$~is automatically incompressible. We also collect the following observation.

\begin{lem}[Incompressibility via connected components]\label{lem.compressiblecomponents}
	A properly embedded compact surface $S$~in a compact smooth $3$-manifold~$M$ is incompressible if and only if all its components are incompressible.
\end{lem}

\begin{proof}
	Clearly, if $S$~is compressible with compressing disc~$D$, then the component of~$S$ that intersects~$\bd D$ also has~$D$ as a compressing disc.
	
	Conversely, suppose $S_0$ is a component of~$S$ that is compressible. A compressing disc~$D$ for~$S_0$ may fail to be a compressing disc for~$S$ because its interior may intersect other components of~$S$. In that case, we first perturb~$D$ slightly to make it transverse to~$S$, and then look at an intersection~$\gamma$ with~$S$ that is innermost in~$D$. Let $D'\sub D$~be a disc bounded by~$\gamma$ (possibly together with an arc in~$\bd S$). If $\gamma$ is an essential curve or arc of~$S$, then $D'$~is a compressing disc for~$S$ and we are done. Otherwise, one can modify~$D$ by replacing~$D'$ with a parallel copy of a disc~$D_S \subset S$ witnessing that $\gamma$~is inessential. The interior of this new compressing disc for~$S_0$ has fewer intersections with~$S$, so an inductive argument finishes the proof.
\end{proof}

Throughout the remainder of this section, $M$~will denote an irreducible and bound\-ary-irreducible oriented compact smooth $3$-manifold.

\begin{dfn}\label{dfn.thurstonnorm}
	Given an orientable compact surface~$S$, we define the non-negative integer
	\[\chi_-(S):= \sum_\text{$C$ component of~$S$} \max\{0, -\chi(C)\}, \]
	where $\chi$~is the Euler characteristic.
	
	For a homology class~$\phi \in \h_2(M, \bd M)$, the \textbf{Thurston norm} of~$\phi$, denoted~$\|\phi\|_M$,  is the minimal value of~$\chi_-(S)$, over all properly embedded surfaces~$S \subset M$ representing~$\phi$. Such a surface~$S$ is said to be \textbf{Thurston norm-realizing} if it realizes this minimum (that is, if~$\|[S]\|_M = \chi_-(S)$) and no union of components of~$S$ represents the zero class in~$\h_2(M, \bd M)$.
\end{dfn}

It is well-known that $\|\cdot\|_M$~extends to a norm on~$\h_2(M, \bd M; \RR)$, which was first observed by Thurston \cite[Theorem 1]{Th86}. We now collect some straightforward facts about Thurston norm-realizing surfaces:

\begin{enumerate}
	\item The only Thurston norm-realizing surface for the class $0 \in \h_2 (M, \bd M)$ is the empty surface.
	
	\item If a properly embedded surface~$S \subset M$ satisfies  $\|[S]\|_M = \chi_-(S)$, one can produce from~$S$ a Thurston norm-realizing surface simply by discarding a maximal null-homologous union of components of~$S$. Each discarded component is necessarily of non-negative Euler characteristic.
	
	\item The fact that $M$~is irreducible and boundary-irreducible implies that properly embedded spheres and discs are null-homologous, so no component of a Thurston norm-realizing surface in~$M$ is a sphere or a disc.
\end{enumerate}

The next property requires a bit more thought, so we promote it to a lemma:

\begin{lem}[Incompressibility of Thurston norm-realizing surfaces]\label{lem.incompressible}
	Every Thurston norm-realizing surface~$S \subset M$ is incompressible.
\end{lem}
\begin{proof}
	Suppose for contradiction that $D$~is a compressing disc for~$S$. We modify~$S$ by removing a small open neighborhood of~$\bd D$ and capping the resulting boundary components with two discs parallel to~$D$. After smoothening, the newly-formed surface~$S'$ is homologous to~$S$ and satisfies $\chi(S')  = \chi(S)+2$. Since $S$~is Thurston norm-realizing, this increase in~$\chi$ cannot amount to a decrease in~$\chi_-$, so $\bd D$~intersects a compressible component~$C$ of~$S$ with non-negative Euler characteristic. But spheres and discs are always incompressible, so~$C$ must be a torus or an annulus. Modifying $C$ by the surgery along~$D$ just described shows that $C$~is homologous to a sphere or a pair of discs, hence null-homologous. This is not allowed by $S$~being Thurston norm-realizing, so we ruled out all possibilities for~$C$, and thus $D$~cannot exist.
\end{proof}

We now introduce the main result in this section.

\begin{dfn}\label{dfn.Tdagger}
	Given a homology class~$\phi \in \h_2(M, \bd M)$, we define $\Td(M,\phi)$~to be the full subcomplex of~$\Sd(M, \phi)$ spanned by the vertices that are Thurston norm-realizing surfaces.
\end{dfn}

\begin{thm}[Connectedness of~$\Td(M, \phi)$]\label{thm.3mfdcase}
	Let $M$~be an irreducible and bound\-ary-irreducible oriented compact smooth $3$-manifold, and let~$\phi\in \h_{2}(M,\bd M)$. Then $\Td(M, \phi)$ is connected. %Moreover, the distance bound given by Corollary~\ref{cor.distbound} for~$\S(M, \phi)$ holds for~$\T(M, \phi)$. 
\end{thm}

In order to prove Theorem~\ref{thm.3mfdcase}, we will establish a distance bound similar to the one given in Proposition~\ref{prop.distbound}. The corresponding statement in this setting is the following, where $\dTd$~denotes the path length distance in the $1$-skeleton of~$\Td(M, \phi)$:

\begin{prop}[Distance bound on surfaces intersecting essentially]\label{prop.3mfddistbound}
	Let $S_0 ,S_1 \subset M$ be a transverse pair of Thurston norm-realizing surfaces representing~$\phi$, and assume moreover that each component of $S_0 \cap S_1$ is essential in both $S_0$~and~$S_1$. If $k$~is a non-negative integer with $|S_0 \cap S_1|<2^k$, then $\dTd(S_1, S_2) \le 2^k$.
\end{prop}

\begin{proof}
	The proof structure is the same as that of Proposition~\ref{prop.distbound}, where we induct over~$k$. Again we will write $C:=S_0 \cap S_1$. The case $k=0$ is immediate.
	
	For the induction step, we will use a similar argument to the one given earlier, where a new surface~$T$ was constructed by decomposing an oriented surgery of $S_0$~and~$S_1$. We need only ensure that $T$~is Thurston norm-realizing and satisfies the additional requirement that for each $i\in\{0,1\}$, every component of $T \cap S_i$ is essential in both $T$~and~$S_i$.
	%As in the proof of Prop~\ref{prop.distbound}, we begin with two transverse Thurston norm-realizing surfaces~$S_0, S_1\subset M$ representing vertices of~$\T(M, \phi)$. Since the~$S_i$ are incompressible by Lemma~\ref{lem.incompressible}, we may use Lemma~\ref{lem.inessential} to additionally assume that all components of the intersection~$C:= S_0 \cap S_1$ are essential in both $S_0$~and~$S_1$.
	
	Performing oriented surgery on the~$S_i$ via the three-step procedure described in the proof of Proposition~\ref{prop.distbound} yields an oriented surface~$\Sigma_0$ representing the class~$2\phi$. For our proof, however, we need an additional step in the construction:
	\begin{enumerate}
		\setcounter{enumi}{3}
		\item Remove a maximal null-homologous union of components of~$\Sigma_0$.
	\end{enumerate}
	Denote the resulting surface by~$\Sigma$. Note that irreducibility and boundary-irreducibility of~$M$ imply that Step 4 removes every sphere or disc component of~$\Sigma_0$ -- although in fact we will soon see that none were present in~$\Sigma_0$ to begin with.
	
	We make three observations concerning $\Sigma_0$~and~$\Sigma$:
	\begin{enumerate}
		\item We have $\chi(\Sigma_0) = \chi(S_0 \sqcup S_1)$. Indeed, as an abstract surface, $\Sigma_0$~can be constructed from the disjoint union~$S_0 \sqcup S_1$ by cutting off small neighborhoods of both copies of~$C$, and gluing them back along the newly formed boundary (Figure \ref{fig.surgery}). This does not alter the Euler characteristic.
		
		\begin{figure}[h]
			\centering
			\def \svgwidth{0.8\linewidth}
			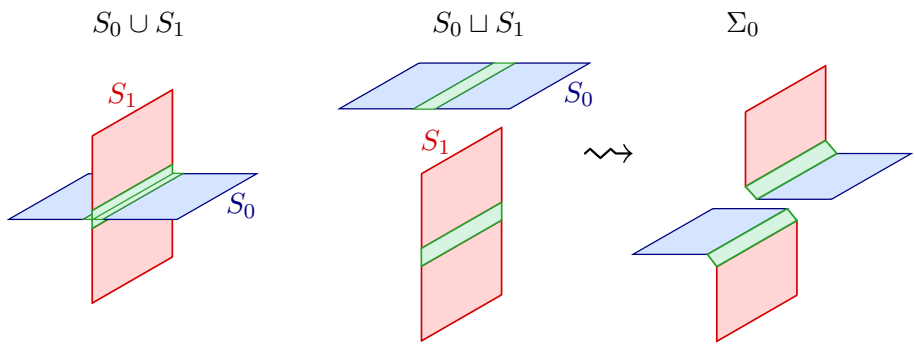
			\caption{Constructing~$\Sigma_0$ as an abstract surface by performing surgery on $S_0 \sqcup S_1$.}
			\label{fig.surgery}
		\end{figure}
		
		\item The surface $\Sigma_0$~has no sphere or disc components, so $\chi(\Sigma_0) \leq \chi(\Sigma)$. To see this, consider the ``seams'' in~$\Sigma_0$ that result from surgery along~$C$. Explicitly, these seams correspond, in the notation of page~\pageref{rampnotation}, to the connected components of~$C\times \left\{ \pm \left(\frac 12,- \frac 12\right)\right\} \subset C\times R$.
		Since there are no sphere or disc components in either of the~$S_i$, any sphere or disc component in~$\Sigma_0$ would have been produced during surgery, and thus have a seam. An innermost seam would then correspond to a component of~$C$ that is inessential in one of the~$S_i$. Such components are excluded by assumption.
		
		\item For each $i\in\{1,2\}$, every component~$\gamma$ of~$\Sigma_0 \cap S_i$ is essential in both $\Sigma_0$~and~$S_i$. Indeed, $\gamma$~is parallel in~$S_i$ to a component of~$C$, and hence essential in~$S_i$. Moreover, $\gamma$~is parallel in~$\Sigma_0$ to a seam, so if it were inessential in~$\Sigma_0$, then $\Sigma_0$~would have an inessential seam. As every innermost inessential seam arises from a component of~$C$ that is inessential in $S_0$~or~$S_1$, we conclude $\gamma$~is essential in~$\Sigma_0$. This observation remains valid when we replace~$\Sigma_0$ with~$\Sigma$.
	\end{enumerate}
	
	Now, the same argument as in the proof of Proposition~\ref{prop.3mfddistbound} shows that $\Sigma$~is the disjoint union of two surfaces~$T_0, T_1$, each representing the class~$\phi$. Step~4 in the construction of~$\Sigma$ ensures that no~$T_m$ contains a union of null-homologous surfaces. Hence, to prove that the~$T_m$ are Thurston norm-realizing, we need only argue that both satisfy $\chi_-(T_m) = \|\phi\|_M$.
	
	Consider the following sequence of (in)equalities:
	\begin{align*}
		\chi_-(T_0)  + \chi_-(T_1) &= -\chi(T_0)-\chi(T_1) & \text{(no spheres or discs in $\Sigma$)}\\
		&= -\chi (\Sigma)\\
		&\le -\chi (\Sigma_0) & \text{(Observation 2)}\\
		&= -\chi (S_0 \sqcup S_1) & \text{(Observation 1)}\\
		&= -\chi(S_0) - \chi(S_1)\\
		&= \chi_-(S_0)+ \chi_-(S_1) & \text{(no spheres or discs in the~$S_i$)}\\
		&= 2 \|\phi\|_M & \text{(the~$S_i$ are Thurston norm-realizing)}&.
	\end{align*}
	Since we cannot have $\chi_-(T_m) <  \|\phi\|_M$ for either~$m$, this shows $\chi_-(T_0) = \chi_-(T_1) = \|\phi\|_M$.
	
	The final step of the proof of Proposition~\ref{prop.distbound} consisted of observing that each component of~$C$ contributes with one component to both~$T_0 \cap S_i$ and none to either~$T_1 \cap S_i$, or vice-versa. In view of Step 4 of the construction of~$\Sigma$, this observation should now be adapted to: each component of~$C$ contributes with \emph{at most} one component to both~$T_0 \cap S_i$ and none to either~$T_1 \cap S_i$, or vice-versa. This change does not affect the conclusion that for some $T\in \{T_0,T_1\}$ we have $|T \cap S_0| = |T \cap S_1| \le \frac{|C|}2$.
	
	The reduction to the induction hypothesis then works as before, provided that for each $i\in\{0,1\}$, every component of $T \cap S_i$ is essential in both $T$~and~$S_i$. But this is a direct consequence of Observation 3 above.
\end{proof}

Before using Proposition~\ref{prop.3mfddistbound} to prove Theorem~\ref{thm.3mfdcase}, we need an auxiliary result to reduce us to the case where all intersections between the~$S_i$ are essential in both~$S_i$. Note that whenever $S_0, S_1$ are transverse properly embedded surfaces in~$M$ that are incompressible, each component of~$S_0 \cap S_1$ that is inessential in one of the~$S_i$ is automatically inessential in the other as well. We will thus simply call a component \textbf{essential} or \textbf{inessential} accordingly, and we will denote by $\ess(S_0, S_1)$ the number of essential components.

\begin{lem}[Removing inessential intersections]\label{lem.inessential}
	Let $S_0, S_1$~be a transverse pair of properly embedded incompressible surfaces in~$M$, and suppose $S_0 \cap S_1$~has an inessential component. Then there exists a properly embedded surface~$S_1'\subset M$ such that:
	\begin{itemize}
		\item $S_1'$~is properly isotopic to and disjoint from~$S_1$,
		\item $S_0$~and~$S_1'$ are a transverse pair, 
		\item $|S_0 \cap S_1'| < |S_0 \cap S_1|$, and
		\item $\ess(S_0, S_1') = \ess(S_0, S_1)$.
	\end{itemize}  
\end{lem}

We point out that we will not need the fourth item in the current section, but it will be useful later, when proving Corollary~\ref{cor.3mfdnodagger}.

\begin{proof}
	%Note first that by incompressibility of $S_0$~and~$S_1$, each component~$\gamma$ of~$S_0 \cap S_1$ that is inessential in one of the~$S_i$ is automatically inessential in the other as well. So let $\gamma$~be one such component, and suppose it is innermost in~$S_0$ among all components of~$S_1 \cap S_0$.
	
	Let $\gamma$~be an inessential component of $S_0 \cap S_1$, and suppose further that it is innermost in~$S_0$ -- that is, it cuts off a disc in~$S_0$ with interior disjoint from~$S_1$.
	
	We consider first the case where $\gamma$~is a circle (Figure~\ref{fig.inessential}, left). For each $i \in \{0,1\}$, let $D_i \subset S_i$ be the disc it bounds in the relevant surfaces (where $D_0$~has interior disjoint from~$S_1$). Since the~$D_i$ intersect only at~$\gamma$, they form an embedded sphere in~$M$, which by irreducibility of~$M$ bounds a ball~$B$. Note that the interior of~$B$ is disjoint from~$S_1$: indeed, $S_1$~is disjoint from the interior of~$D_0$ (because $D_0$~is innermost), and any component of~$S_1$ contained in the interior of~$B$ would be compressible or a sphere. We can thus properly isotope~$S_1$ by pushing it in the direction of its normal bundle that points towards the interior of~$B$, and then use a proper isotopy supported in a small neighborhood of the parallel copy~$D_1'$ of~$D_1$ to push~$D_1'$ across~$B$ and past~$D_0$. For the resulting surface~$S_1'$, the intersection~$S_0\cap S_1'$ is comprised precisely of one parallel copy of each component of~$S_0 \cap S_1$ away from~$D_1$. In particular, $\gamma$~does not contribute to~$S_0 \cap S_1'$, and so $|S_0 \cap S_1'|<|S_0 \cap S_1|$. As any other intersections getting removed are contained in~$D_1$, they are inessential, and so we also have $\ess(S_0, S_1') = \ess(S_0, S_1)$.
	
	\begin{figure}[h]
		\centering
		\def \svgwidth{0.7\linewidth}
		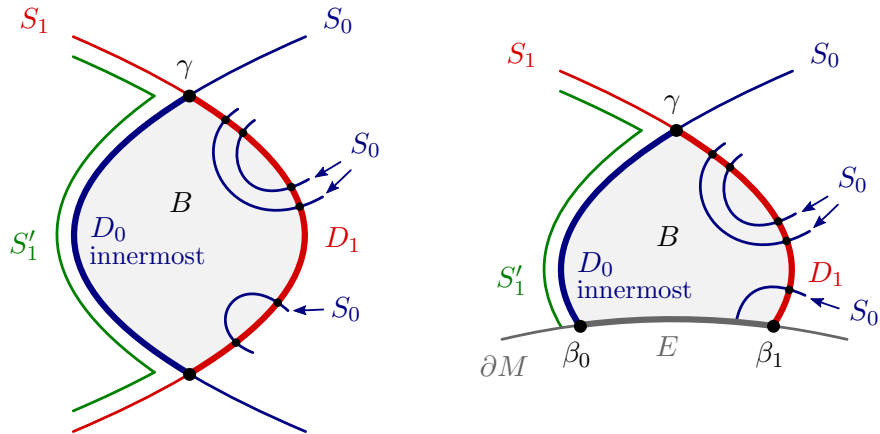
		\caption{Using incompressibility of $S_0$~and~$S_1$ to remove an inessential intersection~$\gamma$ that is innermost in~$S_1$. We depict the cases where $\gamma$~is a circle (left) and where it is an arc (right).}
		\label{fig.inessential}
	\end{figure}
	
	If $\gamma$~is an arc, one proceeds in analogous fashion (Figure \ref{fig.inessential}, right). Since~$\gamma$ is inessential in both~$S_i$, we have:
	\begin{itemize}
		\item a disc~$D_0 \subset S_0$ jointly bounded by $\gamma$~and by an arc~$\beta_0 \subset \bd S_0$, such that the interior of~$D_0$ is disjoint from~$S_1$, and
		\item a disc~$D_1 \subset S_1$ jointly bounded by $\alpha$~and by an arc~$\beta_1 \subset \bd S_1$.				
	\end{itemize}
	Since $D_0 \cup D_1$~is a properly embedded disc in~$M$ with boundary $\beta_0 \cup \beta_1$, boundary-irreducibility of~$\bd M$ guarantees we also have:
	\begin{itemize}
		\item a disc~$E \subset \bd M$ bounded by~$\beta_0 \cup \beta_1$.
	\end{itemize}
	Irreducibility of~$M$ again provides a 3-ball~$B$ with interior disjoint from~$S_0$, and whose boundary is~$D_0 \cup D_1 \cup E$. We push~$S_1$ off of itself in the direction of~$B$ and use~$B$ to properly isotope~$D_0$ through~$D_1$. Again, we have $|S_0 \cap S_1'|<|S_0 \cap S_1|$ and $\ess(S_0, S_1') = \ess(S_0, S_1)$.
\end{proof}

\begin{proof}[Proof of Theorem~\ref{thm.3mfdcase}]
	Let $S_0, S_1$~be Thurston norm-realizing surfaces for~$\phi$, which we aim to show are connected by a path in the $1$-skeleton of~$\Td(M, \phi)$. We may assume that $S_0, S_1$~form a transverse pair -- if not, take a parallel copy of~$S_1$ and perturb it using Proposition~\ref{prop.transverse} to produce another Thurston norm-realizing surface for~$\phi$ disjoint from~$S_1$ and forming a transverse pair with~$S_0$.
	
	Since the~$S_i$ are incompressible by Lemma~\ref{lem.incompressible}, we know that either all components of~$S_0 \cap S_1$ are essential in both~$S_i$, or, by Lemma~\ref{lem.inessential}, there is a surface~$S_1'$ disjoint from~$S_1$, forming a transverse pair with~$S_0$, and such that $|S_0 \cap S_1'| < |S_0 \cap S_1|$. Repeating this argument with $S_1'$~in place of~$S_1$ enough times, we eventually find a surface $S_1^\circ$~connected by a path to~$S_1$ and whose intersections with~$S_0$ are all essential in both $S_0$~and~$S_1^\circ$. Now we apply Proposition~\ref{prop.3mfddistbound} to $S_0$~and~$S_1^\circ$, which finishes the construction of a path from~$S_0$ to~$S_1$.
\end{proof}

\section[$\Sd(M, \phi)$ is simply connected]{The complex $\Sd(M, \phi)$ is simply connected}\label{sec.simpcon}

For this section, we return to the setting where $M$~is an oriented compact smooth manifold of arbitrary dimension~$n$, and $\phi \in \h_{n-1}(M, \partial M)$ is a codimension-$1$ homology class. We will expand the techniques used in the proof of Theorem~\ref{thm.connected} to establish the following result:

\begin{thm}[Simple connectedness of~$\Sd(M, \phi)$]\label{thm.simplyconn}
	Let $M$~be an oriented compact smooth $n$-manifold and let $\phi\in\h_{n-1}(M, \bd M)$. Then the complex~$\Sd(M, \phi)$ is simply connected.
\end{thm}

Similarly to the previous sections, we will start by proving a weaker version of Theorem~\ref{thm.simplyconn} that deals only with collections of hypersurfaces that are transverse, in the sense of Definition~\ref{dfn.transverse}. Subcomplexes of~$\Sd (M, \phi)$ will be called \textbf{transverse} if their set of vertices is transverse (extending the terminology established in Section~\ref{sec.genpos} for finite subsets of~$\V(\Sd(M, \phi))$).

Before presenting this weaker statement (Proposition~\ref{prop.transimplyconn}), we remind the reader of the notions of a simplicial cone, and of a subdivision of a simplicial complex.

\begin{dfn}
	Let $S$~be a simplicial complex. A \textbf{simplicial cone} of~$S$ is a simplicial complex whose vertex set is obtained from~$\V(S)$ by adding one new element $v \not  \in \V(S)$, and whose $k$-simplices are the  $k$-simplices of~$S$ and the $\sigma\cup \{v\}$, with $\sigma$~a $(k-1)$-simplex of~$S$. The new vertex~$v$ is called the \textbf{cone point}.
\end{dfn}

Of course all simplicial cones of~$S$ are isomorphic via a unique isomorphism restricting to the identity on~$S$, so we may talk about ``the'' simplicial cone.

\begin{dfn}
	A \textbf{subdivision} of a simplicial complex~$S$ is a simplicial complex~$T$ with $\V(S) \subseteq \V(T)$, such that there is a homeomorphism $f\colon |S| \to |T|$ satisfying:
	\begin{itemize}
		\item on the (geometrical realization of the) $0$-skeleton of~$S$, the map~$f$ restricts to the inclusion $\V(S) \hookrightarrow \V(T)$, and
		\item for each geometric simplex~$|\tau| := \mathrm{Hull}(\tau) \subseteq |T|$ of~$T$, the pre-image~$f^{-1}(|\tau|)$ is contained in some geometric simplex~$|\sigma|$ of~$S$.
	\end{itemize} 
\end{dfn}

\begin{prop}[Contractibility of $1$-subcomplexes with transverse edges]\label{prop.transimplyconn}
	For each transverse $1$-dimensional subcomplex~$P$ of~$\Sd(M,\phi)$, the inclusion $P \into \Sd(M,\phi)$ extends to some subdivision of the simplicial cone of~$P$.
\end{prop}

Proving Proposition~\ref{prop.transimplyconn} will take most of the present section.
It would be a straightforward affair if we were able to find a vertex~$T$ in~$\Sd(M, \phi)$ such that for each simplex~$\sigma$ of~$P$, the set~$\sigma \cup \{ T\}$ is a simplex of~$\Sd(M, \phi)$. There is however no reason to expect such a hypersurface~$T$ to exist, so we will instead just pick some~$T$ (satisfying a transversality assumption) and ask how far the~$\sigma \cup \{ T\}$ are from being simplices. This approach relies on a family of notions of complexity for simplices in the cone of~$P$, or, more generally, for simplices in a complex~$Q$ with $\V(Q)$~a ``nice enough'' subset of~$\V(\Sd(M, \phi))$.

\begin{dfn}
	Let $Q$~be a nonempty finite simplicial complex such that $\V(Q) \subset \V(\Sd(M, \phi))$ and every simplex of~$Q$ is a transverse set. For each $l \in \NN_{\ge 1}$:
	\begin{itemize}
		\item The \textbf{$l$-complexity} of a $k$-simplex~$\sigma = \{S_0, \ldots, S_k\}$ of~$Q$ is
		\[ \langle\sigma\rangle_l := \begin{cases}
			0 & \text{if $k < l$,}\\
			| S_0 \cap \ldots \cap S_l | & \text{if $k=l$,}\\
			\text{the maximal $l$-complexity among the $l$-dimensional faces of~$\sigma$} & \text{if $k>l$.}
		\end{cases}\]
		\item The \textbf{$l$-complexity} of~$Q$ is
		\[\langle Q \rangle_l := \max \{ \langle \sigma\rangle_l \in \NN \st \text{$\sigma$ is a simplex of~$Q$} \}.\]
	\end{itemize}
\end{dfn}

Note that if $\langle Q\rangle_l=0$ for a certain~$l$, then $\langle Q \rangle_m=0$ for all~$m\ge l$. Moreover, $Q$~is a subcomplex of~$\Sd(M,\phi)$ precisely when $\langle Q\rangle_1=0$.

The definition we just gave is more general than necessary for our purposes. Indeed, as the cone of~$P$ mentioned in the preceding discussion is $2$-dimensional, we will only make use of the notions of $1$- and $2$-complexity.

\begin{prop}[Reducing $2$-complexity]\label{prop.2cpx}
	Let $Q$~be a finite $2$-dimensional simplicial complex with $\V(Q) \subset \V(\Sd(M, \phi))$, and whose simplices are transverse. 
	If $\langle Q\rangle _2>0$, then there is a subdivision~$Q'$ of~$Q$ with $\V(Q') \subset \V(\Sd(M, \phi))$ and transverse simplices, such that
	$\langle Q'\rangle_2 <\langle Q\rangle_2$.
\end{prop}
\begin{proof}
	The proof scheme is similar to that of Proposition~\ref{prop.distbound}. The first step will be to construct, for each $2$-simplex $\sigma = \{S_0, S_1, S_2\}$ of~$Q$ of maximal $2$-complexity~$N$, an oriented surgery~$\Sigma$ of the~$S_i$. Of course we have to clarify what this means for a set of three hypersurfaces. We will then decompose~$\Sigma$ as a disjoint union of three hypersurfaces~$T_0, T_1, T_2$, each representing~$\phi$. Finally, we will use a counting argument to see that adding one of the~$T_m$ to the vertex set of~$Q$ allows us to  subdivide~$\sigma$ into three triangles of smaller $2$-complexity than~$\sigma$.
	
	As before, the construction of the oriented surgery~$\Sigma$ starts with the union~$S_0 \cup S_1\cup S_2$, which we wish to modify near points where the~$S_i$ meet. However, this time we have to consider not only the local model near the intersection of precisely two hypersurfaces  described in the proof of Proposition~\ref{prop.distbound}, but also the neighborhoods of triple points.
	
	Let~$P_0, P_1, P_2$ be the three coordinate planes in~$\RR^3$. Denoting $C:= S_0 \cap S_1 \cap S_2$, the fact that all involved hypersurfaces are oriented implies that there is a neighborhood~$U$ of~$C$ in~$M$ that is diffeomorphic to~$C \times \RR^3$, via a diffeomorphism identifying $S_i \cap U$ with~$C \times P_i$ in a way that preserves all framings (Figure~\ref{fig.triple}, top left).
	
	\begin{figure}[h]
		\centering
		\def \svgwidth{0.75\linewidth}
		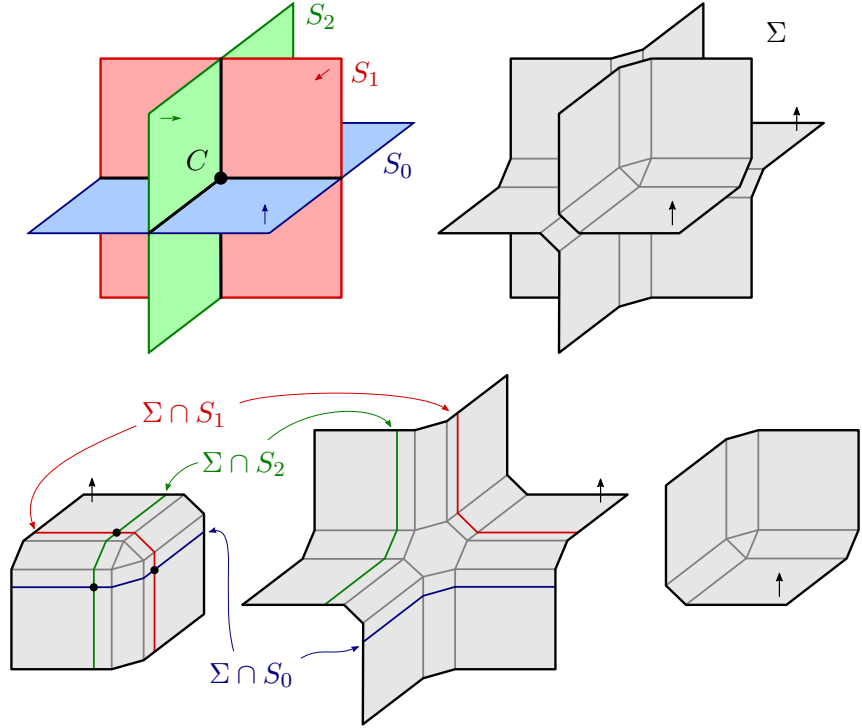
		\caption{The oriented surgery construction near triple points. Top left: the local picture of~$S_0 \cup S_1 \cup S_2$ near~$C$, with framings of the~$S_i$ indicated by arrows. Top right: the oriented surgery~$\Sigma$ of the~$S_i$ near~$C$. Bottom: the three sheets of~$\Sigma$ formed near~$C$, and their intersections with the~$S_i$. The bottom sheet has exactly one triple intersection with each two among the~$S_i$, the middle sheet intersects each~$S_i$ but forming no triple points, and the top sheet is disjoint from all~$S_i$.}
		\label{fig.triple}
	\end{figure}
	
	Whereas the ``ramp'' construction in the proof of Proposition~\ref{prop.distbound} replaces each double intersection with two sheets (Figure~\ref{fig.orientedsum}, top right), its analogue for triple intersections gives rise to three sheets (Figure~\ref{fig.triple}, top right). Performing this modification near triple points, the ramp construction near double points, and then pushing everything in the direction of the framings of the~$S_i$ yields the oriented surgery~$\Sigma$. % (Note that we are also using the fact that the local models for double and triple intersections can be consistently fitted.) 
	The resulting~$\Sigma$ represents the homology class~$3 \phi$ and, together with the~$S_i$, forms a transverse set.
	
	The oriented surgery $\Sigma$~can be decomposed as the disjoint union of three properly embedded hypersurfaces $T_0, T_1, T_2$, each representing~$\phi$:
	As in the proof of Proposition~\ref{prop.distbound}, choose a map $f\colon M \to \SS^1$ for which $f^{-1}(1) = \Sigma$. Regarding $\phi$~as a $1$-dimensional cohomology class, the map~$f$ is a classifying map for~$3\phi$, and so it factors through the $3$-fold cover $\SS^1 \to \SS^1$, $z \mapsto z^3$. Denoting by~$g$ the lifted map, and writing $\omega := \mathrm{e}^{\frac{2\pi}{3} \mathrm{i}}$, we may take $T_0 := g^{-1}(1)$, $T_1 := g^{-1}(\omega)$ and $T_2 := g^{-1}(\omega^2)$.

	At the bottom of Figure~\ref{fig.triple}, we see that the lowest of the three sheets of~$\Sigma$ produced near each component~$c$ of~$C$ forms one triple intersection with each two of the~$S_i$, whereas the other two sheets form no triple intersections. Hence, for exactly one~$m \in \{0,1,2\}$, the component~$c$ contributes with $+1$ to each of the numbers $|T_m \cap S_0 \cap S_1|$, $|T_m \cap S_0 \cap S_2|$ and $|T_m \cap S_1 \cap S_2|$. Therefore, the total number~$N$ of triple intersections among the~$S_i$ decomposes as a sum $N = N_0 + N_1 + N_2$ of non-negative integers, where for each~$m\in \{0,1,2\}$ we have
	\[N_m = |T_m \cap S_0 \cap S_1| = |T_m \cap S_0 \cap S_2| = |T_m \cap S_1 \cap S_2|.\]
	Some~$m$ must then satisfy $N_m \le \frac N3$, so let $T:= T_m$.
	
	We produce the desired subdivision~$Q'$ of~$Q$ as follows: for each triangle~$\sigma$ of~$Q$ with $\langle \sigma\rangle_2 = N$, add the vertex~$T$ just described to~$\V(Q)$, and replace~$\sigma$ with the triangles $\{ T, S_0, S_1\}$, $\{ T, S_0, S_2\}$ and $\{ T, S_1, S_2\}$ (as well as the necessary edges), as illustrated in Figure~\ref{fig.2subdiv}.
	\begin{figure}[h]
		\centering
		\def \svgwidth{0.6\linewidth}
		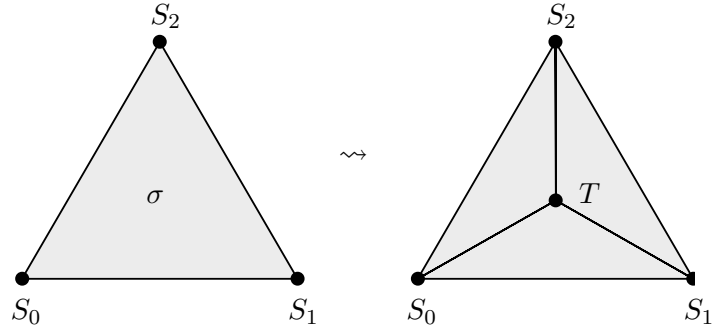
		\caption{$Q$ is subdivided into $Q'$ by replacing each triangle~$\sigma$ of $2$-complexity~$N$ by three triangles with $2$-complexity at most~$\frac N 3$.}
		\label{fig.2subdiv}
	\end{figure}
	
	Each of the new triangles has $2$-complexity at most~$\frac N3 < N$, and so ${\langle Q'\rangle_2 < \langle Q \rangle}$.
\end{proof}

We point out that there seems to be no obstacle to generalizing the above argument beyond complexes of dimension~$2$. More concretely, let $k\in \NN_{\ge 1}$, and let $Q$~be a finite $k$-dimensional simplicial complex with $\V(Q) \subset \V(\Sd(M, \phi))$ and transverse simplices, such that $\langle Q\rangle_k>0$. Then $Q$~can be subdivided into $Q'$~with $\V(Q') \subset \V(\Sd(M, \phi))$ and transverse simplices, such that $\langle Q'\rangle _k <\langle Q\rangle_k$. Indeed, each application of the construction and decomposition of an oriented surgery of the vertices in a $k$-simplex of~$Q$ with maximal $k$-complexity~$N$ would subdivide that simplex into $k+1$~simplices of complexity at most~$\frac{N}{k+1}$.
The main technical annoyance is in showing that the local models such as the one in Figure~\ref{fig.triple}, which for small~$k$ fit our low-dimensional pictures, behave as one would expect when $k$~is larger. As such a statement is unnecessary for our purposes, we will not pursue these details.

\begin{prop}[Reducing $1$-complexity]\label{prop.1cpx}
	Let $Q$~be a finite $2$-dimensional simplicial complex with $\V(Q) \subset \V(\Sd(M, \phi))$, and whose simplices are transverse.
	Suppose ${\langle Q\rangle_2 = 0}$ and ${\langle Q\rangle_1>0}$. Then there is a subdivision~$Q'$ of~$Q$ with $\V(Q') \subset \V(\Sd(M, \phi))$ and transverse simplices, such that $\langle Q'\rangle_2 = 0$ and $\langle Q'\rangle_1 <\langle Q\rangle_1$.
\end{prop}

\begin{proof}
	For each simplex in~$Q$ of maximal $1$-complexity~$N$ (whether it is an edge or a triangle), we perform the oriented surgery construction described in the proofs of Propositions~\ref{prop.distbound} and Proposition~\ref{prop.2cpx} on its vertex set. Note that in the case of three-fold oriented surgeries, there are no triple points, by our assumption that~$\langle Q\rangle_2 = 0$. It will be crucial for our intersection count that when performing the push-off step of the oriented surgeries, the $3$-fold surgeries (that is, the ones originating from triangles) be pushed farther than the two-fold ones (that is, coming from edges), as illustrated in Figure~\ref{fig.pushfarther}.
	
	\begin{figure}[h]
		\centering
		\def \svgwidth{0.7\linewidth}
		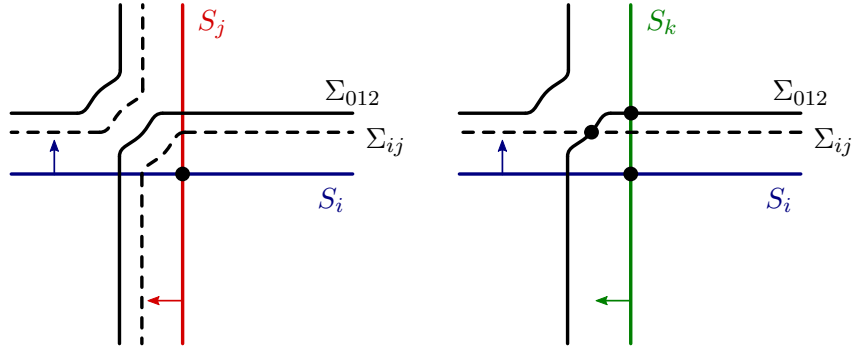
		\caption{The oriented surgery~$\Sigma_{012}$ at a triangle~$\{S_0, S_1, S_2\}$ of~$Q$, and the oriented surgery~$\Sigma_{ij}$ at one of its sides~$\{S_i, S_j\}$. Note that $\Sigma_{012}$~should lie farther from the original surfaces than each~$\Sigma_{ij}$. We depict the intersection pattern between $\Sigma_{012}$~and~$\Sigma_{ij}$ near~$S_i \cap S_j$ (left) and near the intersection of~$S_i$ with the third surface~$S_k$ (right).}
		\label{fig.pushfarther}
	\end{figure}
	
	We wish to decompose these oriented surgeries, producing new hypersurfaces representing~$\phi$ to be added to~$\V(Q')$. Let us begin with the edges.
	
	If $\{S_0, S_1\}$~is an edge of~$Q$ with complexity~$N$, then the corresponding oriented surgery~$\Sigma_{01}$ represents the class~$2\phi$. The argument given in the proof of Proposition~\ref{prop.distbound} shows that one can write $\Sigma_{01} = T_{01}^0 \sqcup T_{01}^1$ with each~$T_{01}^m$ representing~$\phi$, and moreover, for some $m \in \{0,1\}$, we have
	\[|T_{01}^m \cap S_0| = |T_{01}^m \cap S_1| \le \frac N 2.\]
	We take $T_{01}$~to be a~$T_{01}^m$ satisfying the above condition.
	
	The hypersurfaces constructed in this manner over all edges of~$Q$ with complexity~$N$ are now used for subdividing the $1$-skeleton of~$Q$: each such edge~$\{S_0, S_1\}$ is replaced by the edges $\{S_0, T_{01}\}$ and $\{S_1, T_{01}\}$. Thus this subdivision of the $1$-skeleton of~$Q$ has  $1$-complexity strictly lower than~$N$; our goal is to extend it to all of~$Q$.
	
	We argue as in the proof of Proposition~\ref{prop.2cpx} to claim that for each triangle $\sigma := \{S_0, S_1, S_2\}$ of~$Q$ with $\langle \sigma\rangle_1=N$, the oriented surgery~$\Sigma_{012}$ constructed above decomposes as
	\[\Sigma_{012} = T^0_{012} \sqcup T^1_{012} \sqcup T^2_{012},\]
	with each $T_{012}^m$ representing~$\phi$.
	Our goal now is to extend the subdivision of the $1$-skeleton of~$Q$ to all of~$Q$ by choosing an appropriate $m\in \{0,1,2\}$ and adding $T_{012} : = T_{012}^m$ to~$\V(Q')$. The triangles~$\sigma$ will then be replaced with the simplices~$\tau \cup \{T_{012}\}$, where $\tau$~is a simplex of the subdivision of the boundary of~$\sigma$. This is illustrated in Figure~\ref{fig.1subdv}.
	
	\begin{figure}[h]
		\centering
		\def \svgwidth{0.9\linewidth}
		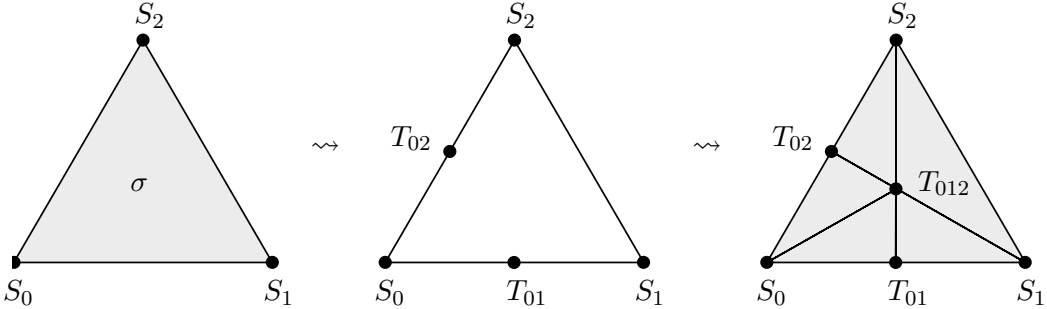
		\caption{Subdividing the simplices in~$Q$ of maximal $1$-complexity~$N$. Left: a triangle~$\sigma = \{S_0, S_1, S_2\}$ with $\langle \sigma\rangle_1 = N$; for this example, we assume $\{S_0, S_1\}$~and~$\{S_0, S_2\}$ are its sides with $1$-complexity~$N$. Middle: we begin by subdividing the $1$-skeleton of~$Q$ by adding ``midpoints'' to the edges of maximal complexity. Right: the subdivision of the $1$-skeleton is extended to all of~$Q$ by replacing each triangle~$\sigma$ with the simplices $\tau \cup \{T_{012}\}$, with $\tau$~ranging over the simplices in the subdivided boundary of~$\sigma$.}
		\label{fig.1subdv}
	\end{figure}
	
	The proof would be complete if we could show that for each triangle~$\sigma = \{S_0, S_1, S_2\}$ of~$Q$, there is a choice of~$m \in \{0,1,2\}$ such that all triangles in the subdivision given by $T_{012} :=T_{012}^m$, as explained above, have $1$-complexity strictly less than~$N$. We will prove a weaker assertion, which will nevertheless suffice:
	
	\begin{claim}
		There is $m\in \{0,1,2\}$ such that all triangles in the subdivision of~$\sigma$ given by $T_{012} :=T_{012}^m$ have $1$-complexity at most~$N$. Moreover, if some side of $\sigma$~has $1$-complexity strictly less than~$N$, then $m$~can be chosen so that all triangles in the subdivision have $1$-complexity strictly less than~$N$.
	\end{claim}
	
	Before proving this claim, we explain how it implies the statement of Proposition~\ref{prop.1cpx}. If all triangles of~$Q$ have a side with $1$-complexity strictly less than~$N$, then the conclusion follows immediately. Suppose then that there exists a triangle~$\sigma$ whose sides all have $1$-complexity~$N$, so the claim does not guarantee a strict decrease in $1$-complexity. The key observation is that this will no longer be true of the triangles formed when subdividing~$\sigma$. Indeed, all these triangles have an edge in the subdivided boundary of~$\sigma$, and those have $1$-complexity strictly less than~$N$. Hence, repeating the whole procedure of subdividing edges and triangles of $1$-complexity~$N$, the resulting subdivision~$Q'$ will satisfy~$\langle Q'\rangle_1 < N$.
	
	\begin{proof}[Proof of the claim]
		All we need to do is show that for some~$m$, all newly-formed edges in the subdivision of~$\sigma$ satisfy the claimed control on $1$-complexity. As already observed, independently of the choice of~$m$, all edges in the subdivided boundary of~$\sigma$ have~$1$-complexity strictly less than~$N$. We therefore need only show that for some~$m$, the edges containing~$T^m_{012}$ satisfy the asserted bound on $1$-complexity. The claim thus follows from the following two statements, which we will now show:
		\begin{enumerate}
			\item Independently of the choice of~$m$, the $1$-complexity of each newly-formed ``short edge'' is bounded by that of the opposite ``long edge''. Formally: let $m\in \{0,1,2\}$, let $\{S_i, S_j\}$ be a side of~$\sigma$ with complexity~$N$, and let $\{i,j,k\} = \{0,1,2\}$ (that is, the $i,j,k$ are three distinct indices). Then we have
			\[|T_{012}^m\cap T_{ij} | \le |T_{012}^m \cap S_{k}|.\]
			%In other words, the $1$-complexity of the ``short'' edge~$\{T_{012}^m, T_{ij}\}$ in the subdivision is bounded by that of the opposite ``long'' edge~$\{T_{012}^m, S_{k}\}$ (independently of the choice of~$m$).
			\item For some $m$, the newly-formed long edges satisfy the claimed bound on $1$-complexity. In other words, there exists $m \in \{0,1,2\}$ such that for all $i \in \{0,1,2\}$ we have 
			\[|T_{012}^m \cap S_i| \le N,\]
			with strict inequality if some side of~$\sigma$ has $1$-complexity strictly less than~$N$.
		\end{enumerate}
		
		For the first item: clearly $T^m_{012}$~only intersects~$T_{ij}$ near points where two among~$S_i$, $S_j$, $S_k$ meet. Thus the proof of this assertion is almost entirely contained in Figure~\ref{fig.pushfarther}. Indeed, the left side of that figure shows that near~$S_i \cap S_j$ there are no components of~$T_{012}^m \cap T_{ij}$, and on the right we see that each component of~$S_i \cap S_k$ contributes with at most one component to~$T_{012}^m \cap T_{ij}$. Explicitly, there is a contribution precisely if the sheet of~$\Sigma_{ij}$ parallel to~$S_i$ belongs to~$T_{ij}$ and the lower sheet of~$\Sigma_{012}$ belongs to~$T^m_{012}$. But when this happens, this component of~$S_i \cap S_k$ also contributes with one component to~$T_{012}^m \cap S_k$. The behavior near components of~$S_j \cap S_k$ is obviously similar. Overall, we conclude that components of~$T_{012}^m \cap T_{ij}$ correspond injectively to components of~$T_{012}^m \cap S_k$, and so $|T_{012}^m \cap T_{ij}| \le |T_{012}^m\cap S_{k}|$.
		
		For the second item: Let $\{i,j,k\} = \{0,1,2\}$. As in the proof of Proposition~\ref{prop.distbound}, the crucial observation is that each component of~$S_i \cap S_j$ contributes with exactly one component to~$T_{012}^m \cap S_i$ and one to~$T_{012}^m \cap S_j$, for precisely one~$m\in \{0,1,2\}$ (refer back to Figure~\ref{fig.orientedsum}, bottom right). Denoting this number of contributions by~$N_{ij}^m$, we thus have
		\[ N_{ij}^0 + N_{ij}^1 + N_{ij}^2 = |S_i \cap S_j| \le N.\]
		
		On the other hand, $|T_{012}^m \cap S_i|$~is the sum of the contributions from~$S_i \cap S_j$ and from~$S_i \cap S_k$. 
		In other words, the $1$-complexity of $\{T^m_{012}, S_i\}$ is given by
		\[|T_{012}^m \cap S_i|= N_{ij}^m + N_{ik}^m.\]
		
		Let us organize all these numbers into a grid, each column corresponding to a choice of~$m$, and each row to a side of~$\sigma$:
		\[\begin{matrix}
			N_{01}^0 & N_{01}^1 & N_{01}^2\\[1em]
			N_{02}^0 & N_{02}^1 & N_{02}^2\\[1em]
			N_{12}^0 & N_{12}^1 & N_{12}^2
		\end{matrix} \]
		The entries in each row add up to at most~$N$, so the sum of all entries on the grid is at most~$3N$, and thus some column~$m$ adds up to at most~$N$. Since the numbers~$|T_{012}^m \cap S_i|$ are precisely the sums of pairs of entries in the $m$-th column, this~$m$ satisfies the first part of item 2. Under the additional hypothesis that some side of~$\sigma$ has $1$-complexity strictly less than~$N$, the corresponding row adds up to strictly less than~$N$, and so the sum of all entries in the grid is strictly less than~$3N$. The same argument then yields~$m$ satisfying the stronger conclusion.  \phantom{\qedhere}
	\end{proof}
	With the claim justified, the proof of Proposition~\ref{prop.1cpx} is complete.
\end{proof}

Having established Propositions \ref{prop.2cpx}~and~\ref{prop.1cpx}, proving Proposition~\ref{prop.transimplyconn}~is straightforward:

\begin{proof}[Proof of Proposition~\ref{prop.transimplyconn}]
	Choose any vertex $T$ of~$\Sd(M, \phi)$, adjusted using Proposition~\ref{prop.transverse} so that $T \not \in \V(P)$ and the set $\V(P) \cup \{T\}$ is transverse. For $Q$~the ($2$-dimensional) simplicial cone of~$P$ with cone point~$T$, it suffices to show that the inclusion of $P \into \Sd(M, \phi)$ extends to some subdivision of~$Q$. By an iterated application of Proposition~\ref{prop.2cpx}, we may subdivide~$Q$ into a simplicial complex~$Q'$ with $\V(Q') \subset \V(\Sd(M, \phi))$ and transverse simplices, such that $\langle Q'\rangle_2=0$. Then an iterated application of Proposition~\ref{prop.1cpx} subdivides $Q'$~into $Q''$~again with $\V(Q'') \subset \V(\Sd(M, \phi))$ and transverse simplices, such that $\langle Q''\rangle_1=0$. By definition this means that all simplices of~$Q''$ are also simplices of~$\Sd(M, \phi)$, so the inclusion $\V(Q'') \into \V(\Sd(M, \phi))$ is a map of simplicial complexes.
\end{proof}

Before establishing simple connectedness of~$\Sd(M, \phi)$, we need an additional observation that reduces us to the setting where the involved hypersurfaces form a transverse set:

\begin{lem} [Perturbing to transverse subcomplexes] \label{lem.transubcpx}
	Every finite subcomplex~$P$ of~$\Sd(M, \phi)$ is isotopic to an isomorphic subcomplex with transverse vertex set.
\end{lem}
\begin{proof}
	The lemma is a consequence of the following statement:
	\begin{claim}
		Let $\mathcal{U} \subset \V(P)$ be a transverse subset and let $S \in \V(P) \setminus \mathcal{U}$. Then there exists a hypersurface $S' \in \V(\Sd(M, \phi)) \setminus \V(P)$ such that:
		\begin{itemize}
			\item the set $\mathcal{U}\cup \{S'\}$ is transverse, and 
			\item the simplicial complex~$P'$ obtained from~$P$ by replacing~$S$ with~$S'$ (in~$\V(P)$ and in all simplices containing~$S$) is a subcomplex of~$\Sd(M, \phi)$ isotopic to~$P$.
		\end{itemize}
	\end{claim}
	
	If we prove this claim, then starting with $\mathcal{U}=\emptyset$, we can iteratively enlarge~$\mathcal{U}$ by replacing each vertex~$S$ of~$P$ by~$S'$ and adding~$S'$ to~$\mathcal U$, ultimately reaching a subcomplex of~$\Sd(M, \phi)$ isotopic to~$P$, whose vertex set is the transverse set~$\mathcal{U}$.
	
	\begin{proof}[Proof of the Claim]
		Let $S_0$ be a push-off of~$S$ disjoint from all the hypersurfaces in~$\V(P)$ that are disjoint from~$S$. We construct~$S'$ by applying Proposition~\ref{prop.transverse} to~$S_0$ and the set~$\mathcal{U}$. Since $S'$~can be chosen to live in an arbitrarily small neighborhood of~$S_0$, we can take $S'$~disjoint from~$S$ and from all hypersurfaces in~$\V(P)$ that are disjoint from~$S$. Hence, for every simplex~$\sigma$ of $P$~containing~$S$, the set~$\sigma':=\sigma \cup \{S'\} \setminus \{S\}$  is a simplex of~$\Sd(M, \phi)$, and thus $P'$~is a subcomplex of~$\Sd(M, \phi)$.
		
		Moreover, $S'$~being disjoint from~$S$ implies that every such $\sigma$, also $\sigma \cup \{S'\}$ is a simplex of~$\Sd(M, \phi)$. Therefore, the path from~$S$ to~$S'$ along the edge $\{S, S'\}$ of~$\Sd(M, \phi)$ extends to an isotopy from each~$\sigma$ to~$\sigma'$, fixing the remaining vertices of~$\sigma$. We extend this as the constant isotopy away from simplices containing~$S$, producing an isotopy from~$P$ to~$P'$. \phantom{\qedhere}
	\end{proof}
	
	Having proved the claim, Lemma~\ref{lem.transubcpx} follows as explained above.
\end{proof}

Theorem~\ref{thm.simplyconn} now follows directly:

\begin{proof}[Proof of Theorem~\ref{thm.simplyconn}]
	By cellular approximation, in order to prove simple connectedness of~$\Sd(M, \phi)$, it suffices to show that the inclusion~$P \into \Sd(M, \phi)$ of every finite $1$-subcomplex~$P$ is null-homotopic. Lemma~\ref{lem.transubcpx} reduces us to the case where~$P$ has transverse vertex set, and Proposition~\ref{prop.transimplyconn} shows that every such~$P$ is null-homotopic in~$\Sd(M, \phi)$.
\end{proof}

\section{Dropping the daggers}\label{sec.dropdagger}

We now discuss the implications of the results in the preceding sections to the complex~$\S(M, \phi)$, defined similarly to~$\Sd(M, \phi)$, except that everything is taken ``up to proper isotopy''.

\begin{dfn}\label{dfn.S}
	Let $M$~be an oriented compact smooth $n$-manifold and let $\phi\in\h_{n-1}(M, \bd M)$. We denote by~$\S(M, \phi)$ the simplicial complex defined as follows:
	\begin{itemize}
		\item The vertices are the (possibly disconnected) oriented smooth properly embedded hypersurfaces in~$M$ representing~$\phi$, up to smooth proper isotopy.
		\item A set of $k+1$ isotopy classes forms a $k$-simplex if those classes admit representatives that are pairwise disjoint.
	\end{itemize}
\end{dfn}

There is an obvious map of simplicial complexes~$p\colon \Sd(M,\phi) \to \S(M, \phi)$ sending each vertex to its isotopy class. This map is clearly surjective on simplices of all dimensions.

If $S_0, S_1 \subset M$ are hypersurfaces representing~$\phi$, we denote by $d_\S(S_0, S_1)$ the path-length distance in the $1$-skeleton of~$\S(M, \phi)$ between their isotopy classes. From the results in Section~\ref{sec.conn} we easily deduce the following corollary:

\begin{cor}[Connectedness of~$\S(M, \phi)$]\label{cor.nodagger}
	Let $M$~be an oriented compact smooth $n$-manifold and let $\phi\in\h_{n-1}(M, \bd M)$. Then $\S(M, \phi)$~is connected. Moreover, if $S_0, S_1$~are transverse representatives of two vertices of~$\S(M, \phi)$ and $k\in \NN$ satisfies $|S_0 \cap S_1|<2^k$, then $\dS(S_0, S_1) \le 2^k$.
\end{cor}
\begin{proof}
	Since $p\colon \Sd(M,\phi) \to \S(M, \phi)$ is surjective on vertices, Theorem~\ref{thm.connected} immediately implies the first part. The second is a direct consequence of Proposition~\ref{prop.distbound}.
\end{proof}

Similarly, for the $3$-dimensional setting, we consider the following simplicial complex:

\begin{dfn}\label{dfn.T}
	Let $M$~be an irreducible and boundary-irreducible $3$-manifold and let $\phi\in\h_{2}(M, \bd M)$. We define~$\T(M, \phi)$ to be the full subcomplex of~$\S(M, \phi)$ spanned by the isotopy classes of Thurston norm-realizing surfaces.
\end{dfn}

Let $d_\T(S_0, S_1)$ denote the distance between the classes of surfaces~$S_0, S_1$ in the $1$-skeleton of~$\T(M, \phi)$. Restricting~$p$ to~$\Td(M, \phi)$ yields a surjective map $\Td(M, \phi) \to \T(M, \phi)$, and we have:

\begin{cor}[Connectedness of~$\T(M, \phi)$]\label{cor.3mfdnodagger}
	Let $M$~be an irreducible and boundary-irreducible oriented compact smooth $3$-manifold, and let $\phi\in\h_2(M, \bd M)$. Then $\T(M, \phi)$ is connected. Moreover, if $S_0, S_1$~are transverse representatives of vertices of~$\T(M, \phi)$ and $k\in \NN$ satisfies  $\ess(S_0 , S_1)<2^k$, then $\dT(S_0, S_1) \le 2^k$.
\end{cor}

We remind the reader that by Lemma~\ref{lem.incompressible}, the~$S_i$ are incompressible and thus each component of~$S_0 \cap S_1$ is essential either in both or in neither of them. The quantity~$\ess(S_0, S_1)$ denotes the number of essential such components.

\begin{proof}
	Surjectivity of the map $\Td(M, \phi) \to \T(M, \phi)$ and the fact that $\Td(M, \phi)$~is connected (Theorem~\ref{thm.3mfdcase}) immediately imply the first part. For the distance bound: by an iterated application of Lemma~\ref{lem.inessential}, we may assume $S_0$~and~$S_1$ intersect only essentially (without changing $\ess(S_0, S_1)$). After this reduction, we have $|S_0 \cap S_1| = \ess(S_0, S_1) < 2^k$ and the statement follows from Proposition~\ref{prop.3mfddistbound}.
\end{proof}

We now turn to the question of whether simple connectedness of~$\Sd(M, \phi)$, established as Theorem~\ref{thm.simplyconn}, descends to~$\S(M, \phi)$. Answering this question will require additional topological input, and we obtain an answer only when $M$ has dimension~$2$.

\begin{cor}[Simple connectedness in the case of surfaces]\label{cor.2mfdcase}
	Let $M$ be an oriented compact smooth surface and let $\phi \in \h_1(M, \bd M)$. Then $\S(M, \phi)$ is simply connected.
\end{cor}

The remainder of this section is dedicated to proving Corollary~\ref{cor.2mfdcase}. Throughout, $M$~ will denote an oriented compact smooth surface, and $\phi\in \h_{1}(M,\bd M)$~a $1$-dimensional homology class.

The key to transferring simple connectedness of~$\Sd(M, \phi)$ to~$\S(M, \phi)$ is the following assertion.

\begin{lem}[Paths within an isotopy class]\label{lem.pathwithinisoclass}
	Let $S, T$ be a transverse pair of oriented properly embedded $1$-dimensional submanifolds of~$M$. If $S$~and~$T$ are in the same proper isotopy class, then there exists a sequence
	\[S = S_0, S_1, \ldots, S_k = T\]
	of oriented properly embedded $1$-dimensional submanifolds of~$M$ all in the same proper isotopy class, such that for each $i \in \{1,\ldots , k\}$, we have $S_{i-1} \cap S_{i} = \emptyset$.
	
	Moreover, if $S$~and~$T$ are part of a transverse family~$\mathcal U$ of properly embedded $1$-dimensional submanifolds of~$M$, then the~$S_i$ can be chosen so that $\mathcal U \cup \{S_1, \ldots, S_{k-1}\}$~is transverse.
\end{lem}

Our proof of Lemma~\ref{lem.pathwithinisoclass} relies on the bigon criterion (Theorem~\ref{thm.bigoncrit} below), a tool available specifically for manifolds of dimension~$2$. 

\begin{dfn}
	Let $M$~be a compact smooth surface, and $\sigma,\tau$~a transverse pair of properly embedded connected $1$-submanifolds of~$M$ (so each of~$\sigma, \tau$ is either a circle or an arc).
	\begin{itemize}
		\item The submanifolds~$\sigma, \tau$ are said to be in \textbf{minimal position} if they cannot be properly isotoped to submanifolds~$\sigma', \tau'$, respectively, such that $|\sigma' \cap \tau'| < |\sigma\cap \tau|$.
		\item We say that $\sigma, \tau$ form a \textbf{bigon} (Figure~\ref{fig.bigon}, left) if there exist two distinct points~$p, q\in \sigma\cap \tau$ and arcs~$\alpha \subset \sigma$, $\beta \subset \tau$ connecting $p$~and~$q$, such that $\alpha\cup\beta$~is a circle (with corners) bounding a disc in~$M$.
		\item We say that $\sigma, \tau$ form a \textbf{half-bigon} (Figure~\ref{fig.bigon}, right) if there exist
		\begin{itemize}
			\item points~$p\in\sigma\cap \tau$, $q_\sigma \in \sigma\cap \bd M$, $q_\tau \in \tau\cap \bd M$,
			\item an arc~$\alpha \subset \sigma$ from~$p$ to~$q_\sigma$,
			\item an arc~$\beta \subset \tau$ from~$p$ to~$q_\tau$, and
			\item an arc~$\gamma \subset \bd M$ from~$q_\sigma$ to~$q_\tau$,
		\end{itemize}
		such that $\alpha\cup\beta\cup \gamma$~is a circle (with corners) bounding a disc in~$M$.
	\end{itemize}
\end{dfn}

\begin{figure}[h]
	\centering
	\def \svgwidth{0.8\linewidth}
	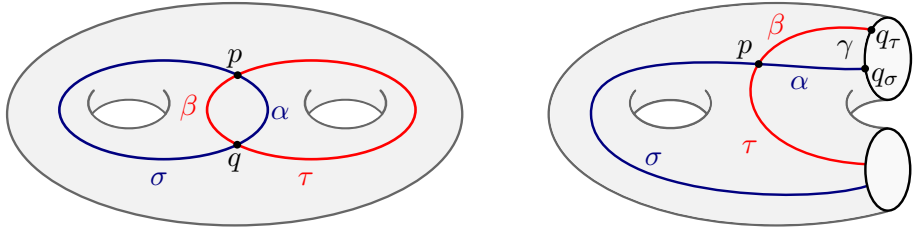
	\caption{Example of a bigon (left) and a half-bigon (right).}
	\label{fig.bigon}
\end{figure}

If $\sigma$~and~$\tau$ form a bigon or a half-bigon, then they are certainly not in minimal position: indeed, after choosing a (half-)bigon that is innermost (meaning, for which the disc in the definition is innermost), we may use it to isotope a small neighborhood of the arc~$\alpha$ past~$\beta$, and then push everything slightly off of~$\sigma$ in the direction away from the bigon. This produces a new $1$-submanifold~$\sigma'$ properly isotopic to~$\sigma$ and having fewer intersections with~$\tau$ (Figure~\ref{fig.bigonfix}). The bigon criterion is a converse to this statement.

\begin{figure}[h]
	\centering
	\def \svgwidth{0.8\linewidth}
	%% Creator: Inkscape 1.0.2 (1.0.2+r75+1), www.inkscape.org
%% PDF/EPS/PS + LaTeX output extension by Johan Engelen, 2010
%% Accompanies image file 'bigonfix.pdf' (pdf, eps, ps)
%%
%% To include the image in your LaTeX document, write
%%   \input{<filename>.pdf_tex}
%%  instead of
%%   \includegraphics{<filename>.pdf}
%% To scale the image, write
%%   \def\svgwidth{<desired width>}
%%   \input{<filename>.pdf_tex}
%%  instead of
%%   \includegraphics[width=<desired width>]{<filename>.pdf}
%%
%% Images with a different path to the parent latex file can
%% be accessed with the `import' package (which may need to be
%% installed) using
%%   \usepackage{import}
%% in the preamble, and then including the image with
%%   \import{<path to file>}{<filename>.pdf_tex}
%% Alternatively, one can specify
%%   \graphicspath{{<path to file>/}}
%% 
%% For more information, please see info/svg-inkscape on CTAN:
%%   http://tug.ctan.org/tex-archive/info/svg-inkscape
%%
\begingroup%
  \makeatletter%
  \providecommand\color[2][]{%
    \errmessage{(Inkscape) Color is used for the text in Inkscape, but the package 'color.sty' is not loaded}%
    \renewcommand\color[2][]{}%
  }%
  \providecommand\transparent[1]{%
    \errmessage{(Inkscape) Transparency is used (non-zero) for the text in Inkscape, but the package 'transparent.sty' is not loaded}%
    \renewcommand\transparent[1]{}%
  }%
  \providecommand\rotatebox[2]{#2}%
  \newcommand*\fsize{\dimexpr\f@size pt\relax}%
  \newcommand*\lineheight[1]{\fontsize{\fsize}{#1\fsize}\selectfont}%
  \ifx\svgwidth\undefined%
    \setlength{\unitlength}{296.34004608bp}%
    \ifx\svgscale\undefined%
      \relax%
    \else%
      \setlength{\unitlength}{\unitlength * \real{\svgscale}}%
    \fi%
  \else%
    \setlength{\unitlength}{\svgwidth}%
  \fi%
  \global\let\svgwidth\undefined%
  \global\let\svgscale\undefined%
  \makeatother%
  \begin{picture}(1,0.27079054)%
    \lineheight{1}%
    \setlength\tabcolsep{0pt}%
    \put(0,0){\includegraphics[width=\unitlength,page=1]{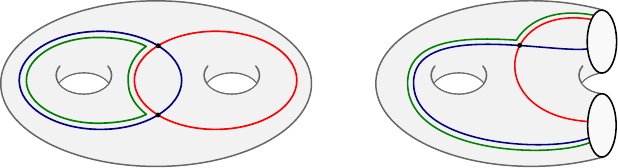}}%
    \put(0.32260268,0.03711671){\color[rgb]{1,0,0}\makebox(0,0)[lt]{\lineheight{1.25}\smash{\begin{tabular}[t]{l}$\tau$\end{tabular}}}}%
    \put(0.17074979,0.03711671){\color[rgb]{0,0,0.50196078}\makebox(0,0)[lt]{\lineheight{1.25}\smash{\begin{tabular}[t]{l}$\sigma$\end{tabular}}}}%
    \put(0.8341399,0.09033181){\color[rgb]{1,0,0}\makebox(0,0)[lt]{\lineheight{1.25}\smash{\begin{tabular}[t]{l}$\tau$\end{tabular}}}}%
    \put(0.73324561,0.07093681){\color[rgb]{0,0,0.50196078}\makebox(0,0)[lt]{\lineheight{1.25}\smash{\begin{tabular}[t]{l}$\sigma$\end{tabular}}}}%
    \put(0.17913848,0.0862268){\color[rgb]{0,0.50196078,0}\makebox(0,0)[lt]{\lineheight{1.25}\smash{\begin{tabular}[t]{l}$\sigma'$\end{tabular}}}}%
    \put(0.80792738,0.21677958){\color[rgb]{0,0.50196078,0}\makebox(0,0)[lt]{\lineheight{1.25}\smash{\begin{tabular}[t]{l}$\sigma'$\end{tabular}}}}%
  \end{picture}%
\endgroup%

	\caption{If two curves~$\sigma,\tau$ form a bigon, one can produce a new curve~$\sigma'$ that is isotopic to~$\sigma$ and has fewer intersections with~$\tau$ (left). Similarly for a half-bigon (right).}
	\label{fig.bigonfix}
\end{figure}
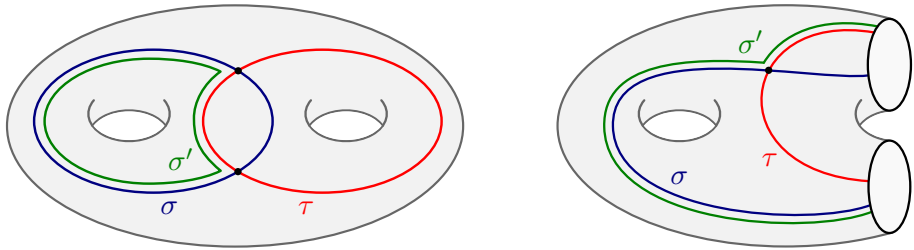

\begin{thm}[The Bigon Criterion]\label{thm.bigoncrit}
	Let $M$~be an oriented compact smooth surface, and let $\sigma,\tau$~be a transverse pair of properly embedded connected $1$-submanifolds of~$M$. If $\sigma$~and~$\tau$ are not in minimal position, then they form a bigon or a half-bigon.
\end{thm}

For a proof of the bigon criterion in the closed case, see the book of Farb and Margalit \cite[Proposition~1.6]{FM12}. The same argument can be adapted to the case of surfaces with boundary.

\begin{proof}[Proof of Lemma~\ref{lem.pathwithinisoclass}]
	Suppose $\sigma, \tau$ are components of $S,T$, respectively, that have non-empty intersection. We first note that $\sigma, \tau$~are not in minimal position: indeed, if we properly isotope~$S$ to~$T$ and then push it slightly along the positive direction of the normal bundle of~$T$, then we will have isotoped~$S$ to be disjoint from~$T$, and thus also~$\sigma$ to be disjoint from~$\tau$. Hence, by the bigon criterion, $\sigma, \tau$ form a (half-)bigon.
	
	By the Diffeotopy Extension Theorem \cite[Theorem~2.4.6]{Wal68}, the isotopy from~$\sigma$ to~$\sigma'$ illustrated in Figure~\ref{fig.bigonfix} can then be extended to an ambient isotopy of~$M$ supported in a small neighborhood of the union of~$\sigma$ and the (half-)bigon. This ambient isotopy restricts to a proper isotopy from~$S$ to a proper $1$-submanifold~$S_1$ disjoint from~$S$. This submanifold~$S_1$ is transverse to~$T$ and satisfies ${|S_1 \cap T| < |S\cap T|}$. Applying Proposition~\ref{prop.transverse}, we may if necessary perturb~$S_1$ to make $\mathcal U \cup \{S_1\}$~a transverse set. Since $S_1$~is transverse to~$T$, a small enough perturbation will not change the topology of~$S_1 \cap T$, and in particular no new intersections between $S_1$~and~$T$ are formed.
	
	This procedure can be iterated until all intersections are removed.
\end{proof}

With this extra topological input, the problem of showing that finite $1$-subcomplexes of~$\S(M, \phi)$ are null-homotopic can be ``lifted along~$p$ to~$\Sd(M, \phi)$'':

\begin{lem}[Lifting $1$-subcomplexes]\label{lem.polylift}
	Let $M$~be an oriented compact smooth surface and let $\phi\in \h_1(M, \bd M)$. Any finite $1$-subcomplex~$P$ of~$\S(M, \phi)$ is the $p$-image of a transverse $1$-subcomplex~$\tilde P$ of~$\tilde{\S}(M,\phi)$ such that the restriction~$p|_{\tilde P}\colon \tilde P \to P$ is a homotopy equivalence.
\end{lem}

\begin{proof}
	Since $p$~is surjective on simplices, each edge of~$P$ may be lifted to an edge in~$\Sd(M,\phi)$ (but adjacent edges of~$P$ might not lift to adjacent edges of~$\Sd(M,\phi)$). Moreover, if we lift one edge at a time and always apply Proposition~\ref{prop.transverse} to the vertices in the lifts, we can guarantee that the union of the lifted edges is a transverse subcomplex of~$\Sd(M,\phi)$.
	
	Whenever two edges of~$P$ share a vertex that is lifted to two distinct vertices of~$\Sd(M,\phi)$, Lemma~\ref{lem.pathwithinisoclass} provides a path contained in~$p^{-1}(v)$ that joins them. Applying this lemma enough times (and always keeping everything transverse), we can construct, for each vertex~$v$ of~$P$, a tree in~$p^{-1}(v)$ connecting the various lifts of~$v$. Take $\tilde P$~to be the finite $1$-subcomplex of~$\Sd(M, \phi)$ comprised of the edge lifts and these trees (Figure~\ref{fig.polylift}). Since $p$~acts on~$\tilde P$ by collapsing each tree to a point, we conclude $p|_{\tilde P}$~is a homotopy equivalence.
\end{proof}

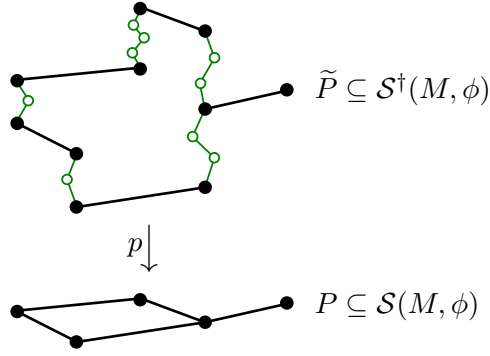
\begin{figure}[h]
	\centering
	\def \svgwidth{0.25\linewidth}
	%% Creator: Inkscape 1.1 (c68e22c387, 2021-05-23), www.inkscape.org
%% PDF/EPS/PS + LaTeX output extension by Johan Engelen, 2010
%% Accompanies image file 'polylift.pdf' (pdf, eps, ps)
%%
%% To include the image in your LaTeX document, write
%%   \input{<filename>.pdf_tex}
%%  instead of
%%   \includegraphics{<filename>.pdf}
%% To scale the image, write
%%   \def\svgwidth{<desired width>}
%%   \input{<filename>.pdf_tex}
%%  instead of
%%   \includegraphics[width=<desired width>]{<filename>.pdf}
%%
%% Images with a different path to the parent latex file can
%% be accessed with the `import' package (which may need to be
%% installed) using
%%   \usepackage{import}
%% in the preamble, and then including the image with
%%   \import{<path to file>}{<filename>.pdf_tex}
%% Alternatively, one can specify
%%   \graphicspath{{<path to file>/}}
%% 
%% For more information, please see info/svg-inkscape on CTAN:
%%   http://tug.ctan.org/tex-archive/info/svg-inkscape
%%
\begingroup%
  \makeatletter%
  \providecommand\color[2][]{%
    \errmessage{(Inkscape) Color is used for the text in Inkscape, but the package 'color.sty' is not loaded}%
    \renewcommand\color[2][]{}%
  }%
  \providecommand\transparent[1]{%
    \errmessage{(Inkscape) Transparency is used (non-zero) for the text in Inkscape, but the package 'transparent.sty' is not loaded}%
    \renewcommand\transparent[1]{}%
  }%
  \providecommand\rotatebox[2]{#2}%
  \newcommand*\fsize{\dimexpr\f@size pt\relax}%
  \newcommand*\lineheight[1]{\fontsize{\fsize}{#1\fsize}\selectfont}%
  \ifx\svgwidth\undefined%
    \setlength{\unitlength}{118.70998953bp}%
    \ifx\svgscale\undefined%
      \relax%
    \else%
      \setlength{\unitlength}{\unitlength * \real{\svgscale}}%
    \fi%
  \else%
    \setlength{\unitlength}{\svgwidth}%
  \fi%
  \global\let\svgwidth\undefined%
  \global\let\svgscale\undefined%
  \makeatother%
  \begin{picture}(1,1.23402435)%
    \lineheight{1}%
    \setlength\tabcolsep{0pt}%
    \put(0.41421882,0.3487035){\color[rgb]{0,0,0}\makebox(0,0)[lt]{\lineheight{1.25}\smash{\begin{tabular}[t]{l}$p$\end{tabular}}}}%
    \put(0,0){\includegraphics[width=\unitlength,page=1]{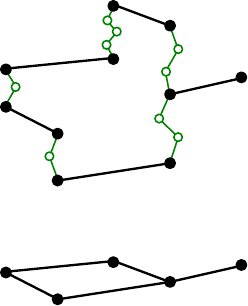}}%
    \put(0.46212396,0.44924064){\color[rgb]{0,0,0}\makebox(0,0)[lt]{\lineheight{1.25}\smash{\begin{tabular}[t]{l}\large \rotatebox{-90}{$\longrightarrow$}\end{tabular}}}}%
    \put(1.07447131,0.89036954){\color[rgb]{0,0,0}\makebox(0,0)[lt]{\lineheight{1.25}\smash{\begin{tabular}[t]{l}$\tilde P \subseteq \Sd(M,\phi)$\end{tabular}}}}%
    \put(1.07447131,0.12775234){\color[rgb]{0,0,0}\makebox(0,0)[lt]{\lineheight{1.25}\smash{\begin{tabular}[t]{l}$P \subseteq \S(M,\phi)$\end{tabular}}}}%
  \end{picture}%
\endgroup%

	\caption{Lifting a finite 1-subcomplex~$P$ in~$\S(M, \phi)$ to~$\Sd(M, \phi)$. We first lift edges using surjectivity of~$p$ (solid dots, thick edges), and then use Lemma~\ref{lem.pathwithinisoclass} to construct trees (hollow dots, thin edges) connecting distinct pre-images of all vertices.}
	\label{fig.polylift}
\end{figure}

\begin{proof}[Proof of Corollary~\ref{cor.2mfdcase}]
	We will show that the inclusion $P \into \S(M, \phi)$ of every finite $1$-subcomplex~$P$ is null-homotopic.
	By Lemma~\ref{lem.polylift}, there is a transverse $1$-subcomplex $\tilde P\subset \Sd(M, \phi)$ making the following diagram commute:
	\[\begin{tikzcd}
		\tilde P \arrow[r, hook] \arrow[d, "p|_{\tilde P}"] & \Sd(M, \phi) \arrow[d, "p"] \\
		P \arrow[r, hook]& \S(M, \phi)
	\end{tikzcd},\]
	and for which the restriction~$p|_{\tilde P}$ is a homotopy equivalence.
	
	Let $g\colon P \to \tilde P$~be a homotopy inverse (at the level of topological realizations), and consider its induced map on cones $\C g \colon \C P \to \C \tilde P$. By Proposition~\ref{prop.transimplyconn}, at the level of topological realizations of simplices, the inclusion $\tilde P\into \Sd(M, \phi)$ factors as 
	\[\tilde P \into \C \tilde P \overset{h}{\to} \Sd(M, \phi),\]
	for some continuous map~$h$. Hence, the following diagram (where all simplicial complexes are to be understood as their topological realizations) commutes up to homotopy:
	\[\begin{tikzcd}
		\tilde P  \arrow[r, hook] \arrow[rr, bend left, hook] & \C \tilde P \arrow[r, "h"] & \Sd(M, \phi) \arrow[d, "p"]\\
		P \arrow[r, hook] \arrow[u, "g"] \arrow[rr, bend right, hook] & \C P \arrow[u, "\C g"]&\S(M, \phi)
	\end{tikzcd}\]
	This shows that the inclusion $P \into \S(M, \phi)$ is homotopic to a map factoring through $P \into \C P$, and thus it is null-homotopic.
\end{proof}

We finish this section by briefly commenting on the only obstacle to extending Corollary~\ref{cor.2mfdcase} to manifolds~$M$ of higher dimension. It all boils down to proving Lemma~\ref{lem.pathwithinisoclass} in higher dimensions. In other words, one would need an affirmative answer to the following question:

\begin{quest}\label{quest.MillionDollarQuest}
	Let $S,T$ be a transverse pair of oriented properly embedded hypersurfaces in an oriented compact smooth manifold~$M$, and suppose $S, T$~are properly isotopic. Does there exist a sequence
	\[S=S_0, S_1, \ldots, S_k = T\]
	of oriented properly embedded hypersurfaces in~$M$, all in the same proper isotopy class, such that $S_{i-1} \cap S_{i} = \emptyset$ for each~$i\in \{1,\ldots,k\}$?
\end{quest}

Our proof in the $2$-dimensional case relies on the bigon criterion (Theorem~\ref{thm.bigoncrit}), whose straightforward generalization to dimensions greater than~$2$ is easily seen to be false. We have however not been able to use counterexamples to give a negative answer to Question~\ref{quest.MillionDollarQuest}. If such counterexamples do exist, we expect that answering this question would be a difficult task, since one would presumably need an invariant that distinguishes between codimension-$1$ submanifolds in the same isotopy class.

\section[Graphs in $2$-complexes]{Co-oriented regular graphs in $2$-complexes}\label{sec.turaev}

We have been dealing with hypersurfaces representing a fixed codimension-$1$ homology class, but through Poincaré Duality one may as well think of them as representing a $1$-dimensional cohomology class. Turaev has described a way of representing $1$-dimensional cohomology classes $\phi\in \h^1(X,\bd X)$ in certain $2$-dimensional CW-complexes~$X$ (relative to their ``boundary subspaces''~$\bd X$) by embedded graphs~$\Gamma$ satisfying some regularity conditions \cite{Tu02}. In this short section, we briefly comment on extensions of the previous results in this article to that context.

More concretely, Turaev treats finite CW-complexes~$X$ of dimension~$2$ that are locally homeomorphic to a cone over a graph, % (called \textbf{$2$-complexes})
and the embedded graphs~$\Gamma$ are required to have closed tubular neighborhoods~$U \iso \Gamma \times [-1,1]$ disjoint from~$\bd X$ (see Sections 1.1~and~1.2 of his paper for details). Such an embedded graph~$\Gamma$ is called \textbf{regular}. A regular graph~$\Gamma$, together with the choice of a component of~$U \setminus \Gamma$ near each component of~$\Gamma$ (called a \textbf{co-orientation}), determines a continuous map $(X, \bd X) \to (\SS^1, \{-1\})$, and thus an element in~$\h^1(X, \bd X)$ (the construction of this map is detailed in Turaev's Section~1.2). Conversely, Turaev shows that all cohomology classes can be obtained in this fashion.

To begin, we ponder whether any two co-oriented regular graphs for a fixed element $\phi \in \h^1(X, \bd X)$ can be connected through sequentially disjoint graphs -- in other words, if the simplicial complex~$\Sd(X, \phi)$ (whose definition is the straightforward adaptation of Definition~\ref{dfn.Sdagger}) is connected. And indeed, all main steps in the proof of Proposition~\ref{prop.distbound} can be translated to this setting, provided one uses the appropriate notion of transversality -- see the proof of Lemma~1.4 in Turaev's article. The co-orientations play the role of the framings of hypersurfaces, and %the well-behaved neighborhoods afford us an adequate notion of ``cutting~$X$ along regular graphs''. 
in fact, an analogue of the oriented sum construction is briefly described therein. The fact that any two vertices~$\Gamma_0, \Gamma_1$ of~$\Sd(X, \phi)$ are connected by a path then follows as in Theorem~\ref{thm.connected} by taking a disjoint parallel copy of $\Gamma_1$ and perturbing it to be transverse to~$\Gamma_0$ (Turaev uses a similar argument in proving his Lemma~1.4).

Regular co-oriented graphs are then used to define a norm on $\h^1(X, \bd X; \RR)$ in much the same way one uses surfaces to define the Thurston norm on~$\h_2(M, \bd M; \RR)$, for $M$~an irreducible and boundary-irreducible oriented compact smooth $3$-manifold. Explicitly, for a regular graph~$\Gamma \subset X$, he writes $\chi_-(\Gamma) := -\chi(\Gamma)$, defines the norm~$\Vert \phi\Vert_X$ of a class~$\phi\in \h^1(X, \bd X)$ to be the minimal~$\chi_-(\Gamma)$ over regular graphs~$\Gamma$ representing~$\phi$, and then extends~$\Vert \cdot \Vert_X$ to~$\h^1(X, \bd X;\RR)$. The regularity assumption on~$\Gamma$ implies that all its components have non-positive Euler-characteristic, making the definition of~$\chi_-$ slightly less cumbersome than in the $3$-manifold case. We call $\Vert \cdot \Vert_X$ the \textbf{Turaev norm}.

As in Section~\ref{sec.thurston}, we can again consider the subcomplex of~$\Sd(X, \phi)$ spanned by regular graphs that realize the Turaev norm of~$\phi$, and for which no union of components represents the zero class. A happy coincidence dictates we also denote this complex by~$\Td(X, \phi)$. When adapted to this setting, the proof of Proposition~\ref{prop.3mfddistbound} not only carries over, but actually becomes simpler: we no longer need to worry about ensuring that the oriented surgery produces no components in~$\Sigma_0$ of positive Euler characteristic, since all regular graphs satisfy $\chi \le 0$ \cite[p.~139]{Tu02}. Concretely, this makes the step where we used Observation~2 become immediate. Moreover, it precludes the need for an analogue of Lemma~\ref{lem.inessential} in adapting the proof of Theorem~\ref{thm.3mfdcase} to show that $\Td(X,\phi)$~is connected.

As our proof of Theorem~\ref{thm.simplyconn} requires no additional topological input, only a more methodical approach to counting intersections between oriented surgeries, Theorem~\ref{thm.simplyconn} can be adapted to show $\Sd(X, \phi)$~is simply connected. In fact, as we are in the $2$-dimensional setting, transversality rules out triple points, obviating the need for an analogue of Proposition~\ref{prop.2cpx}. 

Regarding the results in Section~\ref{sec.dropdagger}, analogues of Corollaries \ref{cor.nodagger}~and~\ref{cor.3mfdnodagger} should also hold for the simplicial complexes $\S(X,\phi)$ and $\T(X,\phi)$, respectively, where regular graphs are taken up to isotopy, with essentially the same arguments. We are however aware of no replacement for the bigon criterion for $2$-complexes that would yield an analogue of Lemma~\ref{lem.polylift}. This seems to be the only obstacle in adapting the proof of Theorem~\ref{cor.2mfdcase} to show that $\S(X, \phi)$~is simply connected.

\section{Tube-equivalence of Seifert surfaces}\label{sec.seifert}

In this section we use the oriented surgery construction to give an alternative proof of the fact that every two Seifert surfaces for a knot~$K$ in a smooth $3$-manifold are tube-equivalent. This theorem is well-established and there are many proofs of various flavors. A Morse-theoretical proof sketch can be found in the lecture notes of Gordon \cite[page 27]{Go78}, a proof using triangulations is in the book of Lickorish \cite[Theorem~8.2]{Li97}, and Scharlemann and Thompson \cite[Theorem 1]{ST88} prove the statement in~$\SS^3$ using minimal surface theory. We start by recalling the basics about tube-equivalences and Seifert surfaces. 

We work in the following setting. Let $M$ be a \textbf{rational homology $3$-sphere}, that is, $M$~is a smooth $3$-manifold for which $\h_{*}(M;\QQ) \iso \h_{*}(\SS^3;\QQ)$ (in particular, $M$~is closed). Moreover, let $K$ be an oriented knot in~$M$, let $\nu(K)$~be an open tubular neighborhood of~$K$, and denote by $E_K:= M\setminus\nu(K)$ the knot exterior.
\begin{dfn}
	A \textbf{Seifert surface} for $K$ is a properly embedded oriented connected surface~$S$ in~$E_K$ such that $\partial S$~is a longitude of~$K$ (with the orientation induced from~$K$).
\end{dfn}

Suppose $S$ is a Seifert surface for~$K$, and let $\alpha \subset E_K$~be an oriented embedded arc such that
\begin{itemize}
\item $\alpha \cap S = \bd \alpha$,
\item $\alpha$ intersects~$S$ transversely,
\item $\alpha$ travels from one side of~$S$ to the same side of~$S$, that is, at one of the endpoints of~$\alpha$ the velocity vector of~$\alpha$ agrees with the framing of~$S$, and at the other endpoint it disagrees.
\end{itemize}
One can modify $S$ by replacing a small neighborhood of~$\bd \alpha$ with the circle bundle of the normal bundle of~$\alpha$, thus producing a new Seifert surface~$S'$ for~$K$. We then say that $S'$~is obtained from~$S$ by \textbf{adding a tube}, and that $S$ is obtained from~$S'$ by \textbf{removing a tube}.

\begin{figure}[h]
	\centering
	\def \svgwidth{0.65\linewidth}
	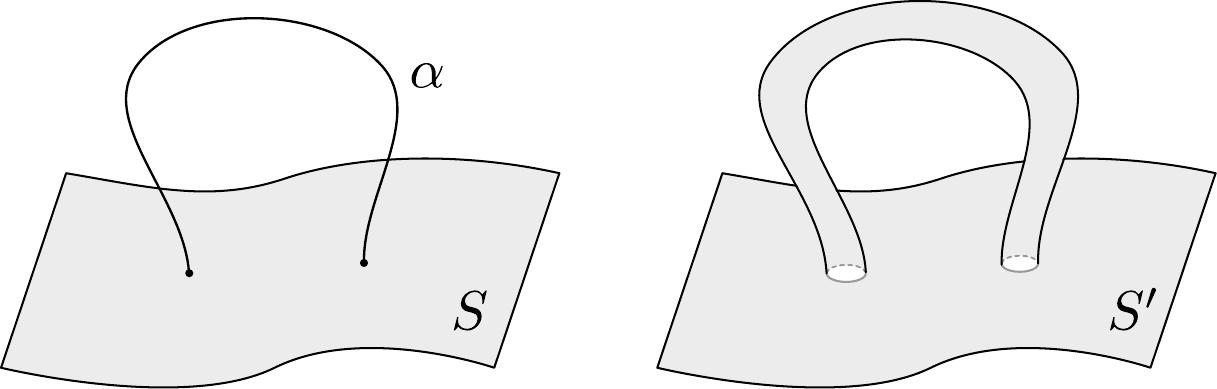
	\caption{The tubing construction. Here, $S'$ is obtained from~$S$ by adding a tube and $S$~is obtained from~$S'$ by removing a tube.}
	\label{fig.tube}
\end{figure}

The following notion of equivalence for Seifert surfaces often plays an important role in the construction of knot invariants.

\begin{dfn}
Two Seifert surfaces for~$K$ are \textbf{tube-equivalent} if one can be obtained from the other by a sequence of addition or tubes, removal of tubes, and proper isotopies.
\end{dfn}

\begin{thm}[Tube-equivalence of Seifert surfaces]\label{thm.tubeq}
	Any two Seifert surfaces for~$K$ are tube-equivalent.
\end{thm}

\begin{proof} 
	We will apply the oriented surgery construction used in the proof of Proposition~\ref{prop.distbound} to show that for every two Seifert surfaces~$S, T$ for~$K$, there is a sequence of Seifert surfaces for~$K$ from~$S$ to~$T$ with each two consecutive ones being disjoint. Then we give a Morse-theoretical argument that establishes tube equivalence of \emph{disjoint} Seifert surfaces.
	
	Using the fact that $M$~is a rational homology $3$-sphere, standard arguments show that $\h_2(E_K, \bd E_K) \iso \ZZ$ (we emphasize that $\ZZ$-coefficients are to be understood). By definition of a Seifert surface, the composition
	\[\h_2(E_K, \bd E_K)  \to \h_1(\bd E_K) \to \h_1(K) \]
	of the connecting homomorphism and the map induced by the projection $\bd E_K \to K$ takes the class represented by each Seifert surface to the standard generator of~$\h_1 (K)$. Hence this composition is an isomorphism and all Seifert surfaces for~$K$ are homologous. It follows that $[\bd S] = [\bd T]$ in~$\h_1(\bd E_K)$, and since homologous curves on a torus are necessarily isotopic, we can, after an isotopy near~$\bd E_K$, assume that $\bd S \cap \bd T = \emptyset$.
	
	We now apply the oriented surgery construction to obtain a sequence of homologous surfaces $S=S_0,\ldots, S_n = T$ in~$E_K$, with $S_i$~disjoint from~$S_{i+1}$ for each $0\le i <n$. It is clear from the construction of the~$S_i$ as pieces of oriented surgeries, and from the fact that $\bd S \cap \bd T = \emptyset$, that each $\bd S_i$ consists of a single isotopic copy of~$\bd S$. Removing all closed components from each~$S_i$ makes it connected, and hence a Seifert surface for~$K$.
	
	To finish the proof, we give an argument for the case $S\cap T = \emptyset$. Since $S$ is connected and non-trivial in homology, we see that $E_K\cutalong S$ is connected, and since $T$ is connected and homologous to~$S$, we obtain that $E_K\cutalong (S\cup T)$~has precisely two connected components. We can think of one such component~$W$ as an oriented ``cobordism with corners'' between $S$~and~$T$ (the precise notion is that of a sutured manifold), and produce a Morse function on~$W$ without extrema, to obtain a handle decomposition with only $1$- and $2$-handles. For a more detailed description of such a construction, see for example the paper by Juhász~\cite[Proof of Theorem~2.13]{Ju06}. These handles correspond to addition and removal of tubes, providing a tube equivalence between $S$~and~$T$.
\end{proof}

\section{Secondary $\ell^2$-torsion}\label{sec:l2torsion}

In this section we give an application of Theorem~\ref{thm.3mfdcase} in defining an invariant of $2$-dimensional homology classes in irreducible and boundary-irreducible compact oriented connected $3$-manifolds with empty or toroidal boundary.

We start by recalling the basics of $\ell^2$-invariants; for proofs we refer to L\"uck's monograph \cite{Lue02}. Let $G$ be a countable group and let $\CC [G]$ be its group algebra over~$\CC$. One defines a scalar product~$(\cdot, \cdot)$  on $\CC [G]$ by setting, for~$g,h\in G$,
\[(g,h) :=\begin{cases}1 & \text{if $g=h$}\\0 &\text{if $g \neq h$} \end{cases},\]
and then extending sesquilinearly. This turns $\CC[G]$ into a pre-Hilbert space, whose completion we denote by~$\ell^2(G)$. The elements of this completion can be described as the (possibly infinite) $\CC$-linear combinations of elements of~$G$ whose coefficients are square-summable in absolute value.

Given a topological space~$W$ with a $G$-action, we consider the chain complex $\C_*^{(2)}(G\curvearrowright W):=\ell^2(G) \otimes_{\ZZ[G]} \C_* (W)$, where~$\C_*(W)$ is the singular chain complex of~$W$ with the standard left $G$-action, and $G$~acts on the right of~$\ell^2(G)$ by right multiplication. More generally, if $V \sub W$~is a $G$-invariant subspace, one can use the singular chain complex~$\C_*(W,V)$ of the pair, and write $\C_*^{(2)}(G\curvearrowright (W, V)):=\ell^2(G) \otimes_{\ZZ[G]} \C_* (W,V)$. The \textbf{$\ell^2$-homology groups} are then defined as usual, except that one mods out \emph{the closures} of the images of the differentials (these are sometimes referred to as \emph{reduced} $\ell^2$-homology groups):
\[\h_i^{(2)}(G\curvearrowright (W, V)):= \ker \partial_i^{(2)} / \overline {\operatorname{im}\partial^{(2)}_{i+1}}.\]

Suppose now that $Z \sub Y \sub X$~is a triple of spaces, such that both $X$~and~$Y$ are path-connected, with universal covering maps $p_X\colon  \tilde X \to X$ and $p_Y \colon \tilde Y \to Y$. We will compare the chain complexes
\begin{align*}
\C ^{(2)}_*(Y,Z) &:= \C_*^{(2)}( \pi_1(Y) \curvearrowright (\tilde Y, p_Y^{-1} (Z))),\\
\C ^{(2)}_*(Y \sub X,Z) &:= \C_*^{(2)}( \pi_1(X) \curvearrowright (p_X^{-1}(Y), p_X^{-1} (Z))),
\end{align*}
whose homologies (always modding out closures of images) we denote by $\h^{(2)}_*(Y,Z)$ and $\h ^{(2)}_*(Y \sub X,Z)$, respectively (notice that only the second one depends on~$X$).

Assume now that $X$~is a finite CW-complex and $Z, Y$~are subcomplexes. If all $\ell^2$-homology groups~$\h^{(2)}_*(Y,Z)$ vanish and $\pi_1(Y)$~satisfies an additional technical condition (for example being residually finite), then one can define a secondary invariant~$\tautwo(Y, Z) \in \RR$, called \textbf{$\ell^{2}$-torsion}  \cite[Definition 3.91]{Lue02}. Similarly, if the groups~$\h ^{(2)}_*(Y \sub X,Z)$ all vanish and $\pi_1(X)$ is residually finite, we can define~$\tautwo(Y \sub X,Z)\in \RR$. In general, a chain complex with well-defined $\ell^2$-torsion is said to be \textbf{$\ell^2$-acyclic}.
We will not need the precise definition of $\ell^2$-torsion, but only the following two properties, which are referred to as induction principle \cite[Theorem 3.35(8)]{Lue02} and multiplicativity of $\ell^2$-torsion \cite[Theorem 3.35(1)]{Lue02}.

\begin{thm}[Properties of $\ell^2$-torsion]\label{thm:l2IndandMult}
Let $X$ be a finite connected CW-complex and $Z\sub Y\sub X$ subcomplexes such that for every basepoint $y\in Y$, the induced map $\pi_1(Y,y)\to\pi_1(X,y)$ is injective. Then $\C_*^{(2)}(Y \sub X,Z)$ is $\ell^2$-acyclic if and only if $\C_*^{(2)}(Y ,Z)$ is $\ell^2$-acyclic, and in that case one has
\[ \tautwo(Y\sub X,Z) =\tautwo(Y,Z). \]
Furthermore, if $\C_*^{(2)}(Y,Z)$ and $\C_*^{(2)}(X,Z)$ are $\ell^{2}$-acyclic,then
\[ \tautwo(X,Z)= \tautwo(X,Y)\cdot \tautwo(Y,Z). \]
\end{thm}

Now we return to $3$-manifolds. Let $M$ be an irreducible and boundary-irreducible compact oriented connected smooth $3$-manifold, and let $S$~be a Thurston norm-realizing surface in~$M$. We will cut~$M$ along~$S$ and denote the submanifolds of~$\bd(M\cutalong S)$ that come from~$S$ by $S_+$~and~$S_-$ depending, respectively, on whether they come from the positive or negative side of~$S$.
The first author showed that the $\ell^2$-torsion of the pair~$(M\cutalong S, S_-)$ is defined \cite[Remark 1.4]{He18}; we refer to the dissertation \cite[Chapter 4]{He19} for more information about $\ell^2$-torsion in this specific context.

\begin{thm}[Invariance of $\ell^2$-torsion for disjoint surfaces]\label{thm:independent}
	Let $M$ be an irreducible oriented compact connected smooth $3$-manifold with empty or toroidal boundary, and let $S, T$ be homologous Thurston norm-realizing surfaces that are disjoint. Then one has
	\[\tautwo(M\cutalong S, S_-)=\tautwo(M\cutalong T, T_-). \]
\end{thm}

We will sketch a proof of this theorem only in the case where $S$~and~$T$ are both connected. An inductive argument reduces the general case to this setting \cite[Lemma~4.10]{He19}.

\begin{proof}[Proof sketch for $S,T$ connected]
	If $S$~and~$T$ are null-homologous, then they are empty and there is nothing to show. Otherwise, we see as in the proof of Theorem~\ref{thm.tubeq} that $M$~is the disjoint union of two $3$-manifolds~$M_0, M_1$, with $\bd M_i = S \cup T \cup  (M_i \cap \bd M)$ for $i \in \{0,1\}$ (see Figure~\ref{fig.cutalongS}).
	
	\begin{figure}[h]
		
		\begin{tikzpicture}
		\begin{scope}[xshift=-3cm]
		\draw (0,0) circle (0.75) ;
		\draw (0,0) circle (1.5);
		
		\draw[dashed] (-1.5,0) -- (-0.75,0);
		\draw[dotted,thick] (1.5,0) -- (0.75,0);
		
		\node (N1) at (0,1.1) {$M_0$};
		\node (N2) at (0,-1.1) {$M_1$};
		
		\node (T) at (1.7,0) {$T$};
		\node (S) at (-1.7,0) {$S$};
		\end{scope}
		\begin{scope}[yshift=-0.35cm]
		
		\draw[dashed] (0,0) -- (0,0.75);
		\draw[dashed] (3,0) -- (3, 0.75);
		\draw[dotted,thick] (1.5,0) -- (1.5,0.75);
		\draw[-] (0,0) -- (3,0);
		\draw[-] (0,0.75) -- (3, 0.75);
		
		\node (NcutS) at (1.5,1.1) {$M\cutalong S$};
		\node (Sm) at (0,-0.3) {$\ S_-$};
		\node (T) at (1.5,-0.3) {$T$};
		\node (Sp) at (3,-0.3) {$S_+$};
		\node (N1) at (0.75,0.35) {$M_0$};
		\node (N2) at (2.25,0.35) {$M_1$};

		\end{scope}
		\end{tikzpicture}
		\caption{Schematic picture for the proof of Theorem~\ref{thm:independent}. On the left one sees the original manifold~$M$. On the right one sees the result after cutting~$M$ along~$S$.}
		\label{fig.cutalongS}
	\end{figure}
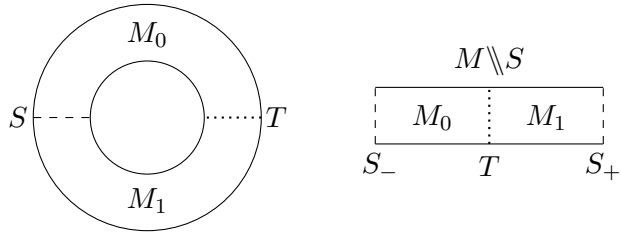
	
	Since Thurston norm-realizing surfaces are incompressible and $2$-sided, it follows from an application of the Loop Theorem \cite[Theorem~1.3.1]{AFW15} to~$M \cutalong T$ that the inclusion $T \into M$ is $\pi_1$-injective. Expressing $\pi_1(M)$ as an HNN-extension of $\pi_1(M \cutalong S)$ then makes it clear that the inclusion $T \into M\cutalong S$ is $\pi_1$-injective, and hence so are also the inclusions of~$T$ into the $M_i$. It now follows from basic facts about amalgams of groups \cite[Chapter~I, Section~1.2]{Se77} that the inclusions of the~$M_i$ into~$M\cutalong S$ are both $\pi_1$-injective, and thus the triples~$S_- \sub M_0 \sub M\cutalong S$ and $ T \sub M_1 \sub M\cutalong S$ satisfy the hypothesis of Theorem~\ref{thm:l2IndandMult}.
	
	Since the complexes $\C_*^{(2)}(M_0, S)$, $\C_*^{(2)}(M_1, T)$ and $\C_*^{(2)}(M\cutalong S, S_-)$ are $\ell^2$-acyclic (we again refer to the dissertation, where this is stated  in the language of taut sutured manifolds \cite[Corollary~3.6]{He19}), Theorem~\ref{thm:l2IndandMult} yields:
	
	\begin{align*}
		\tautwo(M\cutalong S,S_-) & =  \tautwo(M\cutalong S, M_0 ) \cdot \tautwo(M_0, S) & \text{(multiplicativity of $\ell^2$-torsion)}\\
		&= \tautwo(M_1 \sub M\cutalong S, T) \cdot \tautwo(M_0, S) & \text{(we justify this step below)}\\
		&= \tautwo(M_1, T) \cdot \tautwo(M_0, S) &\text{(induction principle)}.
	\end{align*}
	The second equality follows from the fact that the relevant chain complexes are homotopy-equivalent to isomorphic cellular chain complexes, and this preserves $\ell^2$-torsion \cite[Theorem~1.14]{Schi01}.
	
	On the other hand, cutting along~$T$ instead of~$S$ and repeating the argument yields
	\[ \tautwo(M\cutalong T,T_-) = \tautwo(M_1, T) \cdot \tautwo(M_0, S),\]
	from which the result follows.
\end{proof}

As an application of Theorem~\ref{thm.3mfdcase}, we can drop the disjointness assumption.

\begin{cor}[Invariance of $\ell^2$-torsion]\label{cor.l2invariant}
Let $M$ be an irreducible and boundary-irreducible compact oriented connected smooth $3$-manifold with empty or toroidal boundary. If $\phi\in \h_2(M, \bd M)$ is a $2$-dimensional homology class and $S$~is a Thurston norm-realizing surface representing~$\phi$, then the quantity~$\tautwo (M\cutalong S,S_-)$ is independent of~$S$, and is thus an invariant of~$\phi$.
\end{cor}

\begin{proof}
Let $T$~be a different choice of Thurston norm-realizing surface for~$\phi$. By Theorem~\ref{thm.3mfdcase} there is a path in the complex~$\Td(M, \phi)$ of Thurston norm-realizing surfaces from~$S$ to~$T$, and we can inductively apply Theorem~\ref{thm:independent} to every two consecutive surfaces in this path.
\end{proof}

We finish by mentioning that Ben Aribi, Friedl and the first author have related this invariant~$\tautwo (M\cutalong S,S_-)$ to the $\ell^2$-Alexander torsion of~$M$ and (the Poincaré dual of)~$\phi$ \cite[Proposition~4.4]{BFH18}. Moreover, it is known by work of Lück and Schick that if the interior of~$M$ admits a complete finite volume hyperbolic metric, then the $\ell^2$-torsion~$\tautwo(M) = \tautwo(M,\emptyset)$ is, up to a multiplicative constant, the hyperbolic volume \cite[Theorem~0.5]{LS99}. The first author has conjectured that a similar relation holds between $\tautwo(M\cutalong S, S_-)$ and the volume of~$M\cutalong S$, when a $S$~is a taut totally geodesic surface \cite[Conjecture~6.7]{He19}.

\printbibliography

\end{document}